\documentclass{article}

\usepackage{amsmath}
\usepackage{amssymb}
\usepackage{amsthm}
\usepackage{nccmath}		
\usepackage{mathabx}
\usepackage{mathrsfs}
\usepackage{wasysym}
\usepackage{bbm}		
\usepackage{accents}
\usepackage{scalerel}
\usepackage{mathtools}
\mathtoolsset{showonlyrefs=true}		
\usepackage[inline]{enumitem}
\usepackage{ifthen}
\usepackage{braket}			
\usepackage[normalem]{ulem}	
\usepackage{nicefrac}		
\usepackage{cancel}			
\usepackage{xcolor} 
\usepackage[a4paper, left=2.7cm, right=2.7cm, top=3.5cm, bottom=2cm]{geometry}
\usepackage[pdftex=true,hyperindex=true,pdfborder={0 0 0},colorlinks=true,allcolors=blue]{hyperref}
\usepackage{doi}
\usepackage[numbers]{natbib}

\newcommand{\ZZ}{\mathbb{Z}}			
\newcommand{\NN}{\mathbb{N}}			

\newcommand{\isdef}{\coloneqq}			
\DeclarePairedDelimiter\abs{\lvert}{\rvert}		
\DeclarePairedDelimiter\norm{\lVert}{\rVert}	

\newcommand{\PP}{\mathbb{P}}
\newcommand\given[1][]{\mathop{#1\vert}}					
\newcommand{\indicator}[1]{\mathbbm{1}_{#1}}			
\newcommand{\TV}{\mathrm{TV}}			

\newcommand{\symb}[1]{\mathtt{#1}}         
\newcommand{\config}[1]{\underline{#1}} 

\DeclareMathOperator{\maj}{maj}

\newtheorem{theorem}{Theorem} 
\newtheorem{knowntheorem}{Theorem}

\newtheorem{lemma}{Lemma}
\newtheorem{knownlemma}{Lemma}

\newtheorem{precorollary}{Corollary}[theorem]
\theoremstyle{definition}
\newtheorem{definition}{Definition}
\newtheorem{remark}{Remark} 
\newtheorem{example}{Example} 
\newtheorem*{claim*}{Claim}

\NewDocumentEnvironment{corollary}{D<>{}}{
    \if\relax\detokenize{#1}\relax  
    \else
        \ifcsname #1-used\endcsname
            \expandafter\xdef\csname #1-used\endcsname{\the\numexpr\csname #1-used\endcsname+1}%
        \else
            \expandafter\gdef\csname #1-used\endcsname{1}%
        \fi
        \renewcommand{\theprecorollary}{\ref{#1}.\csname #1-used\endcsname}%
    \fi
    \precorollary
}{\endprecorollary}

\AtBeginEnvironment{example}{%
  \pushQED{\qed}%
}
\AtEndEnvironment{example}{\popQED\endexample}
\AtBeginEnvironment{remark}{%
  \pushQED{\qed}%
}
\AtEndEnvironment{remark}{\popQED\endremark}
\AtBeginEnvironment{definition}{%
  \pushQED{\qed}%
}
\AtEndEnvironment{definition}{\popQED\enddefinition}
\newenvironment{claimproof}[1][]{%
    \begin{proof}[\ifthenelse{\equal{#1}{}}{Proof of the claim}{#1}]
    
}{%
    \end{proof}
}

\usepackage{acro}
\DeclareAcronym{iid}{%
    short=i.i.d\acdot,
    long=independent and identically distributed,
    first-style=short
}
\DeclareAcronym{ie}{%
    short=i.e\acdot,
    long=that is,
    first-style=short
}
\DeclareAcronym{eg}{%
    short=e.g\acdot,
    long=for example,
    first-style=short
}
\DeclareAcronym{aka}{%
    short=a.k.a\acdot,
    long=also known as,
    first-style=short
}

\newcommand*\samethanks[1][\value{footnote}]{\footnotemark[#1]}

\allowdisplaybreaks

\usepackage{tikz}
\usetikzlibrary{calc,tikzmark,shapes,arrows}
\usetikzlibrary{decorations.pathreplacing}
\usepackage{pgfplots}		

\tikzstyle{site}=[inner sep=0pt,thick,circle,draw=black,fill=black,minimum size=3pt]

\newlength{\graphunit}
\newcommand{\graphG}[5]{
    \begin{scope}[xshift=#1,yshift=#2,xscale=#3,yscale=#4]
        \node[site] (#5a) at (-1\graphunit,0) {};
        \node[site] (#5b) at (+1\graphunit,0) {};
        \draw (#5a) -- (#5b);
    \end{scope}
}
\newcommand{\graphGG}[5]{
    \begin{scope}[xshift=#1,yshift=#2,xscale=#3,yscale=#4]
        \graphG{0}{2\graphunit}{1}{1}{#5a}
        \graphG{0}{-2\graphunit}{1}{-1}{#5b}
        \node[site] (#5c) at (0,0) {};
        
        \draw (#5aa) to[bend right=15] (#5ba);
        \foreach \X in {#5aa,#5ab,#5ba,#5bb} {
            \draw (#5c) -- (\X);
        }
    \end{scope}
}
\newcommand{\graphGGG}[5]{
    \begin{scope}[xshift=#1,yshift=#2,xscale=#3,yscale=#4]
        \graphGG{-4\graphunit}{0}{1}{1}{#5a}
        \graphGG{4\graphunit}{0}{-1}{1}{#5b}
        \node[site] (#5d) at (0,0.5\graphunit) {};
        \node[site] (#5e) at (0,-1.5\graphunit) {};

        \draw (#5d) -- (#5e);
        \draw (#5aaa) to[bend left] (#5bab);
        \draw (#5aab) to[bend left] (#5baa);
        \foreach \X in {#5aaa,#5aab,#5aba,#5abb,#5baa,#5bab,#5bba,#5bbb,#5ac,#5bc} {
            \draw (#5d) -- (\X);
        }
    \end{scope}
}
\newcommand{\graphGGGG}[5]{
    \begin{scope}[xshift=#1,yshift=#2,xscale=#3,yscale=#4]
        \graphGGG{0}{+3\graphunit}{1}{1}{#5a}
        \graphGGG{0}{-3\graphunit}{1}{-1}{#5b}

        \draw (#5aaaa) to[bend right] (#5baaa);
        \draw (#5ae) -- (#5bbbb);
        \draw (#5abbb) -- (#5be);
    \end{scope}
}
\newcommand{\graphGGGGG}[5]{
    \begin{scope}[xshift=#1,yshift=#2,xscale=#3,yscale=#4]
        \graphGGGG{-8\graphunit}{0}{1}{1}{#5a}
        \graphGGGG{+8\graphunit}{0}{-1}{1}{#5b}
        \begin{scope}[yshift=0.5\graphunit]
            \node[site] (#5f) at (90:1\graphunit) {};
            \node[site] (#5g) at (162:1\graphunit) {};
            \node[site] (#5h) at (234:1\graphunit) {};
            \node[site] (#5i) at (306:1\graphunit) {};
            \node[site] (#5j) at (18:1\graphunit) {};
        \end{scope}

        \draw (#5f) -- (#5g) -- (#5h) -- (#5i) -- (#5j) -- (#5f);
        \draw (#5aabab) -- (#5babab);
        \draw (#5abbc) -- (#5bbbaa);
        \draw (#5abbaa) -- (#5bbbc);
    \end{scope}
}

\tikzstyle{stateO}=[inner sep=0pt,circle,draw=black,fill=white,minimum size=4pt]
\tikzstyle{stateI}=[inner sep=0pt,circle,draw=black,fill=black,minimum size=4pt]
\def\rnconfigI{%
    {1, 1, 0, 1, 0, 1, 0, 0, 1, 1, 0, 1, 1, 1, 1, 1, 0, 0, 1, 1, 0, 0, 1, 0, 1, 0, 1, 1, 0, 0, 0, 0, 0, 0, 0, 1, 0, 1, 0, 0, 1, 0, 1, 0, 1, 0, 0, 0, 0, 1, 1, 0, 1, 1, 1, 0, 1, 0, 0, 0, 1, 0, 0, 0, 0, 1, 1, 0, 0, 0, 1, 1, 0, 0, 0, 0, 0, 0, 0, 0, 1}
}
\def\rnconfigII{%
    {1, 1, 0, 1, 1, 0, 1, 0, 1, 0, 0, 0, 0, 0, 1, 0, 1, 1, 1, 0, 0, 0, 0, 1, 0, 0, 0}
}
\def\rnconfigIII{%
    {1, 1, 1, 0, 0, 1, 0, 0, 0}
}
\def\rnconfigIV{%
    {1, 0, 0}
}

\newlength{\xplottick}
\title{%
    Zero-one laws for events with positional symmetries\thanks{%
        This research was conducted during the 2023 Summer Research Camp in Mathematics at the American University of Beirut (AUB).
		The authors thank the Center for Advanced Mathematical Sciences (CAMS) and the Department of Mathematics at AUB for their support.
		S.~Taati also gratefully acknowledges the support of CAMS through a CAMS Fellowship.
		CAMS ORCID: 0009-0004-5763-5004.
	}
	\footnotetext{Last update:~\today}
}

\author{%
	Yahya Ayach\thanks{%
		Department of Mathematics, American University of Beirut, Beirut, Lebanon.
	}
	\and
	Anthony Khairallah\samethanks[2]
	\and
	Tia Manoukian\samethanks[2]
	\and
	Jad Mchaimech\samethanks[2]
	\and
	Adam Salha\thanks{%
		Suliman S. Olayan School of Business (OSB), American University of Beirut, Beirut, Lebanon.
	}
	\and
	Siamak Taati\samethanks[2]~\thanks{%
		Center for Advanced Mathematical Sciences (CAMS), American University of Beirut, Beirut, Lebanon.
	}
}

\date{}

\begin{document}

\maketitle

\begin{abstract}
    We use an information-theoretic argument due to O'Connell (2000) to prove that every sufficiently symmetric event concerning a countably infinite family of independent and identically distributed random variables is deterministic (\ac{ie}, has a probability of either~$0$ or~$1$).  The \ac{iid} condition can be relaxed.
    This result encompasses the Hewitt-Savage zero-one law and the ergodicity of the Bernoulli process, but also applies to other scenarios such as infinite random graphs and simple renormalization processes.
    
\end{abstract}

\section{Introduction}
\label{sec:intro}

\subsection{Background and main results}
\label{sec:intro:results}

Much of probability theory concerns how deterministic phenomena emerge from randomness.
Zero-one laws refer to theorems asserting that the occurrence or non-occurrence of certain events are almost surely pre-determined without identifying whether they actually occur or not.
The most famous zero-one law is the Kolmogorov zero-one law, which states that every tail event concerning a sequence of independent random variables has a probability of either $0$ or~$1$~\cite[Appx.]{Kolmogorov1956}.
In the current paper, we generalize two other classic zero-one laws: the Hewitt-Savage zero-one law~\cite{HS1955}, and the ergodicity of the Bernoulli process (see e.g., the monograph by Walters~\cite[Thm.~1.12]{Walters1982}).

The latter two theorems establish the determinism of certain symmetric events concerning a sequence of independent and identically distributed (\ac{iid}) random variables.
Let $M$ be a measurable space where the random variables take their values.
An event $E\subseteq M^{\NN}$ is said to be \emph{exchangeable} if it is invariant under all finitary permutations $\pi\colon\NN\to\NN$ of the indices, that is, $\config{a}=(a_1,a_2,\ldots)\in E$ if and only if $\config{a}^\pi\isdef (a_{\pi(1)},a_{\pi(2)},\ldots)\in E$.  Here, being \emph{finitary} means that $\pi$ keeps all but finitely many elements unchanged.

\begin{knowntheorem}[Hewitt-Savage zero-one law]
\label{thm:01law:hewitt-savage}
	Let $\config{X}=(X_1,X_2,\ldots)$ be an infinite sequence of \ac{iid} random variables with values in a measurable space~$M$.
    Then, for every exchangeable event $E\subseteq M^\NN$, we have $\PP(\config{X}\in E)\in\{0,1\}$.
\end{knowntheorem}

For the other theorem, we state the version that concerns bi-infinite sequences of random variables.
An event $E\subseteq M^{\ZZ}$ is said to be \emph{shift-invariant} if $\config{a}=(\ldots,a_{-1},a_0,a_1,\ldots)\in E$ if and only if $\config{a}^\sigma\in E$, where $\config{a}^\sigma\isdef(\ldots,a_0,a_1,a_2,\ldots)$ is the sequence obtained by shifting $\config{a}$ by one position towards left.

\begin{knowntheorem}[Ergodicity of the Bernoulli process]
\label{thm:01law:shift-invariant}
	Let $\config{X}=(\ldots,X_{-1},X_0,X_1,\ldots)$ be a bi-infinite sequence of \ac{iid} random variables with values in a measurable space~$M$.
    Then, for every shift-invariant event $E\subseteq M^\ZZ$, we have $\PP(\config{X}\in E)\in\{0,1\}$.
\end{knowntheorem}

Observe that the hypothesis in both theorems is the invariance of the event in question under a group of transformations that only permute the elements of the sequence.  These transformations are examples of what we call positional symmetries.
\begin{definition}[Positional symmetry]
	Let $\Gamma$ be a countable set.
	A \emph{positional symmetry} of an event $E\subseteq M^\Gamma$ is a map $\pi\colon\Gamma\to\Gamma$ such that $\config{a}=(a_k)_{k\in\Gamma}\in E$ if and only if $\config{a}^\pi\isdef (a_{\pi(k)})_{k\in\Gamma}\in E$.
\end{definition}
\noindent Let us emphasize that, despite the name ``symmetry,'' we do not a priori require $\pi$ to be bijective, or even injective.
Our first theorem states that if an event concerning a countable family of \ac{iid} random variable has enough injective positional symmetries, then it is deterministic.

\begin{theorem}[Zero-one law for \ac{iid} RVs and symmetric events]
\label{thm:main:iid}
    Let $\config{X}\isdef (X_k)_{k\in\Gamma}$ be a countable family of random variables with values from a measurable space~$M$.
    Consider an event $E\subseteq M^\Gamma$, and suppose that for every finite $J \subseteq \Gamma$, the set $E$ has an injective positional symmetry $\pi\colon\Gamma\to\Gamma$ such that $\pi(J) \cap J = \varnothing$.
    Then, 
    $\PP\big(\config{X}\in E\big)\in\{0, 1\}$.
\end{theorem}

Theorems~\ref{thm:01law:hewitt-savage} and~\ref{thm:01law:shift-invariant} are both corollaries of Theorem~\ref{thm:main:iid}.
Indeed, for every finite set $J\subseteq\NN$, there exist (many) finitary permutations $\pi$ that map $J$ to a set disjoint from~$J$.  Likewise, for every finite $J\subseteq\NN$ one can find (many) translations $n+J$ that are disjoint from~$J$.
On the other hand, Theorem~\ref{thm:main:iid} covers examples to which Theorems~\ref{thm:01law:hewitt-savage} and~\ref{thm:01law:shift-invariant} do not apply.  We discuss two applications of this theorem in Section~\ref{sec:applications}.  First, we show that if $G$ is a sufficiently symmetric countably infinite graph, then the graph-theoretic properties of the random subgraphs of $G$ satisfy a zero-one law (Section~\ref{sec:app:graphs}).  Second, we show that every tail event concerning the iterates of a simple sufficiently symmetric renormalization map on an \ac{iid} sequence of random variables is deterministic (Section~\ref{sec:app:renormalization}).

In the language of ergodic theory, the Hewitt-Savage zero-one law asserts that every homogeneous product measure on a countable product space is (invariant and) ergodic for the group of finitary permutations, while Theorem~\ref{thm:01law:shift-invariant} says the same thing for the group generated by the shift map.
Call a semigroup $\Pi$ of maps on a countable set~$\Gamma$ \emph{dispersive} if for every finite set $J\subseteq\Gamma$, there exists an element $\pi\in\Pi$ such that $\pi(J)\cap J=\varnothing$.
Now, Theorem~\ref{thm:main:iid} can be rephrased as follows:
\begin{quote}
    \emph{Every homogeneous product measure on a countable product space is ergodic with respect to every dispersive semigroup of injective positional maps.}
\end{quote}
As observed in the previous paragraph, the group of finitary permutations and the group generated by the shift transformation are both dispersive.

The Hewitt-Savage zero-one law goes hand in hand with de~Finetti's theorem, which states that the distribution of every exchangeable sequence of random variables (taking values in a standard Borel space) is a mixture of homogeneous product measures~\cite{AB2015,HS1955}.  An exchangeable sequence is a sequence whose distribution is invariant under finitary permutations.  In light of the uniqueness of ergodic decomposition (see the paper by Dynkin~\cite{Dynkin1978}), de~Finetti's theorem can be thought of as a converse to the Hewitt-Savage theorem: every measure that is ergodic for the group of finitary permutations is a homogeneous product measure.
The analogous converses for Theorems~\ref{thm:01law:shift-invariant} and~\ref{thm:main:iid} do not hold: the product measures are by no means the only ergodic invariant measures for the shift transformation.


A proof of Theorem~\ref{thm:main:iid} is given in Section~\ref{sec:main-result:proof}.
Our proof is an adaptation of an information-theoretic proof of the Hewitt-Savage theorem given by O'Connell~\cite{OConnell2000} (see Section~\ref{sec:intro:oconnel}).  Although a more mainstream measure-theoretic proof (along the same lines as the standard proofs of Theorems~\ref{thm:01law:hewitt-savage} and~\ref{thm:01law:shift-invariant}) is also possible, we find the information-theoretic argument more insightful.  Furthermore, the information-theoretic approach makes the proofs of the following results somewhat less technical.

The \ac{iid} hypothesis in Theorem~\ref{thm:main:iid} can be relaxed at the cost of adding an extra condition on how the positional maps affect the distribution of the random variables.  In Section~\ref{sec:non-iid}, we prove a general theorem in this direction (Theorem~\ref{thm:main:non-iid}).  However, the hypothesis of the latter theorem is rather technical.  Below, we state two more concrete corollaries of this general result, one in which the random variables are allowed to have different distributions, and one in which they are allowed to be only asymptotically independent.

\begin{theorem}[Zero-one law for independent but not identically distributed RVs]
\label{thm:main:independent}
    Let $\config{X}\isdef (X_k)_{k\in\Gamma}$ be a countable family of independent random variables with values from a measurable space~$M$, where $X_k$ is distributed according to a measure $p_k$.
    Consider an event $E\subseteq M^\Gamma$, and suppose that for every finite set $J\subseteq\Gamma$ and every $\delta>0$, the event $E$ has an injective positional symmetry $\pi\colon\Gamma\to\Gamma$ such that $\pi(J)\cap J=\varnothing$ and $\sum_{k\in\Gamma}\norm{p_{\pi(k)}-p_k}_\TV<\delta$.
    Then, 
    $\PP\left(\config{X}\in E\right)\in\{0,1\}$.
\end{theorem}

\begin{theorem}[Zero-one law for asymptotically independent RVs]
\label{thm:main:k-mixing}
    Let $\config{X}\isdef (X_k)_{k\in\Gamma}$ be a countable family of random variables with values from a measurable space~$M$, and suppose that $\config{X}$ has the Kolmogorov mixing property.
    Consider an event $E\subseteq M^\Gamma$, and suppose that for every finite set $J\subseteq\Gamma$, the event $E$ has a positional symmetry $\pi\colon\Gamma\to\Gamma$ such that $\pi(J)\cap J=\varnothing$ and $\config{X}$ and $\config{X}^\pi=(X_{\pi(k)})_{k\in\Gamma}$ have the same distribution.
    Then, 
    $\PP\left(\config{X}\in E\right)\in\{0,1\}$.
\end{theorem}

\noindent
The notation $\norm{q-p}_\TV$ stands for the total variation distance between two measures $p$ and $q$.
The Kolmogorov mixing condition is one of the variants of the intuitive concept of being ``asymptotically independent''.
It is satisfied, for instance,
by every recurrent countable-state Markov chain with deterministic initial state~\cite{BF1964},
every extremal Gibb random field~\cite{Georgii2011}, and every determinantal process~\cite{Lyons2003}.
The definition of the Kolmogorov mixing property can be found in Section~\ref{sec:non-iid:k-mixing}.
The proofs of Theorems~\ref{thm:main:independent} and~\ref{thm:main:k-mixing} are given in Sections~\ref{sec:non-iid:independent} and~\ref{sec:non-iid:k-mixing}, respectively.

Various other extensions of the Hewitt-Savage zero-one law can be found in the literature.
In the case where the random variables are independent but not necessarily identically distributed, Aldous and Pitman gave a necessary and sufficient condition on the marginal distributions of the random variables in order for them to satisfy the Hewitt-Savage zero-one law~\cite{AP1979} (see also~Georgii's ``canonical'' monograph~\cite[Thm.~4.23]{Georgii1979}).
The hypothesis on the marginal distributions in our Theorem~\ref{thm:main:independent} is certainly not necessary (see Remark~\ref{rem:main:independent:not-necessary} in Section~\ref{sec:non-iid:independent}), but when restricted to exchangeable events, Theorem~\ref{thm:main:independent} covers an earlier result of Blum and Pathak~\cite{BP1972}.

Blackwell and Freedman proved an analogue of the Hewitt-Savage zero-one law for any recurrent countable-state Markov chain that has a deterministic initial state~\cite{BF1964}.  This result was obtained by applying the original Hewitt-Savage law to the regenerative representation of the Markov chain.  As noted by James, the same argument shows that, in fact, every \emph{partially exchangeable} event regarding such a Markov chain is deterministic~\cite[Intro.]{James1996}.  The notions of partial exchangeability for events and processes were introduced by Diaconis and Freedman, who first proved the latter zero-one law as a byproduct of their de~Finetti-type result for recurrent partially exchangeable processes~\cite{DF1980}.  A complementary result for transient partially exchangeable processes was recently proved by Halberstam and Hutchcroft~\cite{HH2024}.


The original works of de~Finetti, Hewitt and Savage and most of the subsequent developments around exchangeability were primarily motivated by the foundations of statistics.
In a parallel development, Georgii and Preston proved a de~Finetti-type theorem 
in the context of statistical mechanics~\cite{Georgii1976,Preston1979,Georgii1979}.
Namely, they showed that, under general conditions, every canonical Gibbs measure for a given interaction can be represented as a mixture of (grand canonical) Gibbs measures for the same interaction and arbitrary values of the activity parameters.
In particular, every extremal canonical Gibbs measure is an extremal Gibbs measure for some values of the activity parameters, and vice versa.  It is easy to show that exchangeable events are deterministic under extremal canonical Gibbs measures.  As a corollary, they thus obtained a Hewitt-Savage zero-one law for extremal Gibbs measures.
(The hypotheses of these results are, for instance, satisfied when the alphabet is finite, the set of sites is $\ZZ^d$ for some $d\in\NN$, and the interaction is finite-range and shift-invariant.)

Note that, although extremal Gibbs measures and recurrent Markov chains with deterministic initial states satisfy the Kolmogorov mixing property, the Blackwell-Freedman and Georgii-Preston zero-one laws do not directly follow from our Theorem~\ref{thm:main:k-mixing}, because when $\pi$ is a finitary permutation, $\underline{X}$~and $\underline{X}^\pi$ rarely have the same distribution.  We conjecture that, in Theorem~\ref{thm:main:k-mixing}, the requirement that $\underline{X}$~and $\underline{X}^\pi$ have the same distribution can be replaced with a much weaker requirement.


For \emph{finite} families of random variables, none of the above-mentioned results hold exactly, but Diaconis and Freedman obtained an approximate version of de~Finetti's theorem~\cite{Diaconis1977,DF1980a}.  Deriving approximate zero-one laws appears to be a more subtle task, for even simple examples of purely exchangeable events concerning a finite family of \ac{iid} random variables can fail to be approximately deterministic.
Nonetheless, Fagin's zero-one law for the first order properties of Erdős-Rényi random graphs~\cite{Fagin1976} (see the monograph by Spencer~\cite{Spencer2001}) and the sharp threshold results for the monotonic Boolean functions with symmetries~\cite{Russo1982,Talagrand1994,FK1996} can be thought of as approximate (or asymptotic) zero-one laws that are driven by symmetries.

For more on the many aspects of exchangeability and related topics, see the reviews by Kingman~\cite{Kingman1978}, Aldous~\cite{Aldous1985}, 
Kallenberg~\cite{Kallenberg2005}, and Austin~\cite{Austin2008}, and the references therein.

Applications of information-theoretic tools and arguments in probability theory and statistics have a long and rich history.
Notable examples include (but are not limited to) Rényi's proof of the Markov chain convergence theorem~\cite{Renyi1961}, proofs of the central limit theorem by Linnik, Barron and others~\cite{Linnik1959,Barron1986,ABB+2004}, and a general treatment of large deviations~\cite{Csiszar1984}.
See the books by Johnson~\cite{Johnson2004}, Csiszár and Shields~\cite{CS2004}, and Cover and Thomas~\cite[Chap.~11]{CT2006}, and the recent article by Gavalakis and Kontoyiannis~\cite{GK2023} for a review of some of such applications.
Recently, Gavalakis, Kontoyiannis, and Berta came up with information-theoretic proofs of the approximate version of de~Finetti's theorem for finite families~\cite{GK2021,GK2023,BGK2024}.
It would be very interesting to see if information-theoretic techniques can shed light on approximate or asymptotic zero-one laws (for finite families of random variables) such as the ones mentioned above.

\subsection{Information theory and O'Connell's argument}
\label{sec:intro:oconnel}
The standard proofs of the Kolmogorov and Hewitt-Savage zero-one laws and the ergodicity of the Bernoulli process are based on measure theory (see \ac{eg}, the books by Dudley~\cite[Sec.~8.4]{Dudley2004} or Klenke~\cite{Klenke2020}).  
Neil O'Connell came up with a short and elegant information-theoretic proof of the Hewitt-Savage zero-one law that limited the measure-theoretic aspect to a simple intuitive lemma~\cite{OConnell2000}.
An almost verbatim adaptation of O'Connell's argument gives a proof of the ergodicity of the Bernoulli process.
Our first result (Theorem~\ref{thm:main:iid}) formulates a general hypothesis that makes the same argument work.  The relaxation to non-\ac{iid} families (Theorem~\ref{thm:main:non-iid} and its corollaries, Theorems~\ref{thm:main:independent} and~\ref{thm:main:k-mixing}) combines O'Connell's idea with an approximation argument.

Let us now briefly review O'Connell's proof of the Hewitt-Savage theorem.
O'Connell presented his proof in the elementary case in which the random variables take their values from a finite set.
We present a variant of the proof that works in general and still remains mostly elementary.
For a review of basic concepts in information theory, we refer to Chapter~2 of the book by Cover and Thomas~\cite{CT2006}.
As usual, we use the notation $H(X)$ for the Shannon entropy of a random variable~$X$, and $I(X:Y)$ for the mutual information between two random variables $X$ and $Y$.

The argument relies on two lemmas.  The first is a simple but useful information-theoretic inequality.  The second is a standard fact from measure theory.

\begin{knownlemma}[O'Connell's inequality]
\label{lem:OConnell}
	Let $V,W_1,W_2,\ldots,W_n$ be discrete random variables and suppose that $W_1,W_2,\ldots,W_n$ are independent. Then,
	\useshortskip
	\begin{align}
		H(V) &\geq \sum_{i=1}^{n} I(V:W_i) \;.
	\end{align}
\end{knownlemma}


\begin{figure}
    \begin{center}
        \begin{tikzpicture}[baseline,>=stealth,xscale=0.8, yscale=0.6]
            \begin{scope}
                \clip (0,0) node {$H(V)$} circle (3);

                \fill[shift=(140:3.5),rotate=-40,blue!20] (0,0) circle (1.7 and 0.85);
                \fill[shift=(200:3.5),rotate=20,blue!20] (0,0) circle (2 and 1);
                \fill[shift=(-80:3),rotate=-80,blue!20] (0,0) circle (1 and 1.05);
                \fill[shift=(-10:4),rotate=-10,blue!20] (0,0) circle (2.3 and 1.8);
                \fill[shift=(55:3.2),rotate=55,blue!20] (0,0) circle (0.9 and 0.9);
            \end{scope}
            
            \draw[very thick] (0,0) node {$H(V)$} circle (3);
                
            \draw[shift=(140:3.5),rotate=-40,very thick] (0,0) circle (1.7 and 0.85)
                node[above left=2em] {$H(W_1)$} node[above right=1em,xshift=-0.5em,blue] {\scriptsize $I(V:W_1)$};

            \draw[shift=(200:3.5),rotate=20,very thick] (0,0) circle (2 and 1)
                node[below left=1.8em] {$H(W_2)$} node[below right=1.7em,xshift=-2.3em,blue] {\scriptsize $I(V:W_2)$};

            \draw[shift=(-80:3),rotate=-80,very thick] (0,0) circle (1 and 1.05)
                node[below=1.5em] {$H(W_3)$} node[left=2.2em,yshift=-0.5em,blue] {\scriptsize $I(V:W_3)$};

            \draw[shift=(-10:4),rotate=-10,very thick] (0,0) circle (2.3 and 1.8)
                node[below=3em] {$H(W_4)$} node[xshift=-1em,yshift=3.7em,blue] {\scriptsize $I(V:W_4)$};

            \draw[shift=(55:3.2),rotate=55,very thick] (0,0) circle (0.9 and 0.9)
                node[above=1.4em,xshift=0.5em] {$H(W_5)$} node[xshift=-3.7em,yshift=1.2em,blue] {\scriptsize $I(V:W_5)$};
        \end{tikzpicture}
    \end{center}
\caption{%
    An intuitive illustration of O'Connell's inequality.
    Each region represents the information carried by one of the random variables.
    The regions corresponding to independent random variables are depicted to be disjoint.
    This illustration however must not be taken too literally because, although $W_1,W_2,\ldots,W_n$ are independent, the information pieces they carry about $V$ need not be independent.
}
\label{fig:OConnell:inequality}
\end{figure}
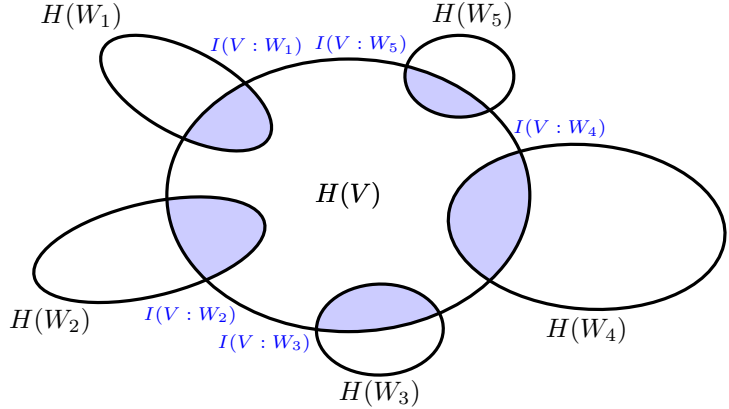


\begin{knownlemma}[Closedness of independence]
\label{lem:independence}
	Let $Y,X_1,X_2,\ldots$ be random variables, and suppose that for every $m$, the random variable $Y$ is independent of $(X_1,X_2,\ldots,X_m)$.  Then, $Y$ is independent of the entire sequence $(X_1,X_2,\ldots)$.
\end{knownlemma}


\noindent See Figure~\ref{fig:OConnell:inequality} for an illustration of the statement of Lemma~\ref{lem:OConnell}.  In Section~\ref{sec:non-iid}, we will prove an approximate version of this lemma for almost independent random variables.  Lemma~\ref{lem:independence} is the main step in the proof of the Kolmogorov zero-one law (see Remark~\ref{rem:01law:kolmogorov}).

The proof of the theorem now proceeds as follows.  Let $Y\isdef\indicator{E}(X_1,X_2,\ldots)$, where $\indicator{E}$ is the indicator of~$E$.
Let $C\subseteq M$ be a measurable set and set $Z_k\isdef\indicator{C}(X_k)$ for each $k$.
Since $X_1,X_2,\ldots$ are independent, by Lemma~\ref{lem:OConnell}, for every $n$, we have
\begin{align}
    H(Y) &\geq I(Y:Z_1) + I(Y:Z_2) + \cdots + I(Y:Z_n) \;.
\end{align}
Since $E$ is exchangeable and $X_i$'s are identically distributed, we have $I(Y:Z_1)=I(Y:Z_2)=\cdots=I(Y:Z_n)$.  Thus, $H(Y)\geq n I(Y:Z_1)$.
But $H(Y)$ is bounded by $\log 2$ and $n$ is arbitrary, hence $I(Y:Z_1)=0$.
Therefore, $Y$ is independent of $Z_1$.  Since $C$ is arbitrary, it follows that $Y$ is independent of~$X_1$.
Similarly, by grouping $X_i$'s into blocks of size~$m$ and applying the same argument, we can show that $Y$ is independent of $(X_1,X_2,\ldots,X_m)$ for every~$m$.
Lemma~\ref{lem:independence} now implies that $Y$ is independent of the entire sequence $(X_1,X_2,\ldots)$.
In particular, $Y$ is independent of itself.  Hence,
\begin{align}
    \PP\big(\config{X}\in E) = \PP\big(\config{X}\in E\cap E) = \PP\big(\config{X}\in E)\PP\big(\config{X}\in E) \;,
\end{align}
which implies that $\PP\big(\config{X}\in E)$ is either $0$ or~$1$.

\begin{remark}[Information-theoretic ``proof' of the Kolmogorov zero-one law]
\label{rem:01law:kolmogorov}
    The second half of the above argument (starting from the application of Lemma~\ref{lem:independence}) simply repeats the standard proof of the Kolmogorov zero-one law.  Let us point out that this part, too, has an information-theoretic explanation:
    Since $Y$ is independent of $(X_1,X_2,\ldots,X_m)$, we have $I\big(Y:(X_1,X_2,\ldots,X_m)\big)=0$.
    Letting $m\to\infty$, we thus obtain $I\big(Y:(X_1,X_2,\ldots)\big)=0$.
    Now, applying the data-processing inequality, we can write
    \begin{align}
        H(Y) &= I(Y:Y) = I\big(Y:\indicator{E}(X_1,X_2,\ldots)\big)
            \leq I\big(Y:(X_1,X_2,\ldots)\big) = 0 \;.
    \end{align}
    Therefore, $Y$ is deterministic, which means $\PP\big(\config{X}\in E)$ is either $0$ or~$1$.

    Nevertheless, the use of measure theory in making this argument precise seems inevitable.
    First, we need to make sense of $I(Y:W)$ when $W$ is not necessarily discrete.
    For the current purpose, we can either define $I(Y:W)\isdef H(Y)-H(Y\given W)$ or use the more general definition as a supremum over finite observations~\cite{Pinsker1964}.
    Second, we have to justify the identity $\lim_{m\to\infty}I\big(Y:(X_1,X_2,\ldots,X_m)\big)=I\big(Y:(X_1,X_2,\ldots)\big)$.  This follows from the martingale convergence theorem (alternatively, see Pinsker's monograph~\cite[Sec.~2.2]{Pinsker1964}).
\end{remark}

\section{Proof of the main result}
\label{sec:main-result:proof}

Given $\config{a}=(a_k)_{k\in\Gamma}$ and $\pi\colon\Gamma\to\Gamma$, we write $\config{a}^\pi$ to denote the family~$(a_{\pi(k)})_{k\in\Gamma}$.  For $K\subseteq\Gamma$, the notation $\config{a}_K$ stands for the subfamily $(a_k)_{k\in K}$.
The notation $X\sim Y$ means $X$ and $Y$ have the same distribution.

\begin{proof}[Proof of Theorem~\ref{thm:main:iid}]
    Let $Y\isdef\indicator{E}(\config{X})$.
    It is enough to show that $Y$ is independent of $\config{X}_K$, for every finite $K\subseteq\Gamma$.  The result would then follow as discussed in Section~\ref{sec:intro:oconnel}.

    So, let $K\subseteq\Gamma$ be finite.
    In order to be able to apply O'Connell's reasoning,
    let us show that $E$ has a sequence $\pi_1,\pi_2,\ldots$ of injective positional symmetries such that the sets $K,\pi_1(K),\pi_2(K),\ldots$ are disjoint.
    We choose $\pi_1,\pi_2,\ldots$ recursively.  Let $\pi_0$ be the identity map on~$\Gamma$.  Having found $\pi_1,\ldots,\pi_{n-1}$, we set $J\isdef\bigcup_{i=0}^{n-1}\pi_i(K)$ and let $\pi_n$ be an injective positional symmetry of $E$ such that $\pi(J)\cap J=\varnothing$.

    Next, let $C\subseteq M^K$ be a measurable set.
    Since $K,\pi_1(K),\pi_2(K),\ldots$ are disjoint, the random variables
    \useshortskip
    \begin{align}
        \indicator{C}(\config{X}),\indicator{C}(\config{X}^{\pi_1}),\indicator{C}(\config{X}^{\pi_2}),\ldots
    \end{align}
    are independent.
    Furthermore,
    \begin{align}
        I\big(Y:\indicator{C}(\config{X}^{\pi_i}_K)\big)=I\big(Y:\indicator{C}(\config{X}_K)\big)
    \end{align}
    for all~$i$.  Indeed, for every $a,b\in\{0,1\}$, we have
    \begin{align}
        \PP(\indicator{E}(\config{X})=a, \indicator{C}(\config{X}_K)=b)
            &= \PP(\indicator{E}(\config{X}^{\pi_i})=a, \indicator{C}(\config{X}^{\pi_i}_K)=b)
            && \text{(because $\config{X}^{\pi_i}\sim\config{X}$)} \\
        &= \PP(\indicator{E}(\config{X})=a, \indicator{C}(\config{X}^{\pi_i}_K)=b)
            && \text{(because $\pi_i$ is a symmetry of $E$)}
    \end{align}
    from which the equality of the mutual informations follows.
    (Here, $\config{X}^{\pi_i}\sim\config{X}$ because $\pi_i$ is injective and $X_k$'s are \ac{iid}.)
    
    Now, by O'Connell's inequality (Lemma~\ref{lem:OConnell}), for every $n$, we have
    \begin{align}
        H(Y) &\geq \sum_{i=1}^n I\big(Y:\indicator{C}(\config{X}^{\pi_i}_K)\big)
            = n I\big(Y:\indicator{C}(\config{X}_K\big) \;.
    \end{align}
    But $H(Y)$ is finite and $n$ is arbitrary, hence $I\big(Y:\indicator{C}(\config{X}_K)\big)=0$.
    Therefore, $Y$ is independent of~$\indicator{C}(\config{X}_K)$.
    Since $C$ is arbitrary, it follows that $Y$ is independent of $\config{X}_K$, as claimed.
    This concludes the proof.
\end{proof}

\section{Applications}
\label{sec:applications}

In this section, we present corollaries of Theorem~\ref{thm:main:iid} in two concrete scenarios: infinite random graphs and renormalization maps.

\subsection{Zero-one law for infinite random graphs}
\label{sec:app:graphs}

Finitary permutations disregard almost any potential structure on the index set.
The power of Theorem~\ref{thm:main:iid} can be illustrated by allowing additional structure on the index set and restricting the required symmetries to those respecting that structure.
Theorem~\ref{thm:01law:shift-invariant} assumes the structure of the group~$\ZZ$ on the index set.
(Note: Theorem~\ref{thm:01law:shift-invariant} remains valid if we replace $\ZZ$ with any other countable group.)
Let us now consider the scenario in which the index set has a graph structure.

We will focus on simple graphs.
A \emph{graph-theoretic property} $P$ is a property that is invariant under graph isomorphisms, that is, if $G$ and $H$ are two isomorphic graphs, then $G$ satisfies $P$ if and only if $H$ does.
For instance, the property of admitting a spanning subtree with maximum degree~$3$ is graph-theoretic.
By a \emph{random subgraph} of a graph $G=(V,E)$ with \emph{parameter $p$}, we mean a random graph obtained from $G$ by removing each edge of $E$ with probability~$1-p$ and keeping it with probability~$p$, independently of the other edges.  For instance, a random subgraph of a complete graph is nothing but an Erdős-Rényi random graph.

\begin{corollary}<thm:main:iid>[Zero-one law for sufficiently symmetric infinite random graphs]
\label{cor:01law:random-subgraph}
    Let $G=(V,E)$ be a countably infinite graph.  Suppose that for every finite set $A\subseteq V$ of vertices, $G$ has an automorphism~$\theta$ such that $\theta(A)\cap A=\varnothing$.  Then, for every $p\in[0,1]$, every graph-theoretic property of a random subgraph of $G$ with parameter~$p$ is deterministic.
\end{corollary}
\begin{proof}
    Consider the family $\config{X}=(X_e)_{e\in E}$ of Bernoulli random variables, where $X_e$ indicates whether the edge $e$ is kept or removed.
    By assumption, these random variables are \ac{iid}. 

    Let $J\subseteq E$ be a finite set.  Let $A\subseteq V$ denote the set of vertices that are incident to $J$.  Clearly, $A$ is also a finite set.
    Let $\theta$ be an automorphism of $G$ satisfying $\theta(A)\cap A=\varnothing$, and define $\pi\colon E\to E$ by $\pi(ab)\isdef\theta(a)\theta(b)$ for every $ab\in E$.  Note that $\theta(a)\theta(b)\in E$ if and only if $ab\in E$ because $\theta$ is an automorphism of $G$.
    Furthermore, $\pi$ is injective because $\theta$ is.
    Lastly, $\pi(J)\cap J=\varnothing$ because the vertices incident to $J$ and $\pi(J)$ are disjoint.
    
    Now, consider a graph-theoretic property $P$.  Restricted to the subgraphs of~$G$, we can view $P$ as a (measurable) subset of $\{0,1\}^E$.
    Since $\theta$ is an automorphism of $G$, a subgraph $H$ of $G$ satisfies $P$ if and only if its image $\theta(H)$ satisfies~$P$.
    It follows that for every $(x_e)_{e\in E}\in\{0,1\}^E$, we have $(x_e)_{e\in E}\in P$ if and only if $(x_{\pi(e)})_{e\in E}\in P$, that is, $\pi$ is a positional symmetry of $P$.
    Therefore, according to Theorem~\ref{thm:main:iid}, the subgraph represented by $(X_e)_{e\in E}$ satisfies $P$ either with probability~$0$ or with probability~$1$.
\end{proof}

When $G$ is highly symmetric, the zero-one law of Corollary~\ref{cor:01law:random-subgraph} sometimes follows from a much stronger uniqueness property.

\begin{example}[Infinite Erdős-Rényi random graph]
\label{exp:graph:infinite-erdos-renyi}
    An infinite Erdős-Rényi random graph with parameter~$p$ is a random subgraph, with parameter~$p$, of the complete graph $K_\NN$ on~$\NN$.
    The complete graph $K_\NN$ clearly satisfies the hypothesis of Corollary~\ref{cor:01law:random-subgraph}.
    However, the resulting zero-one law is a trivial consequence of the well-known fact that, every two (non-degenerate) infinite Erdős-Rényi random graphs in the same probability space are almost surely isomorphic (see \ac{eg}, the survey paper by Cameron~\cite{Cameron2013}). 
    %
\end{example}

\noindent The uniqueness principle mentioned in the latter example holds for other examples of infinite random graphs such as the universal triangle-free graph~\cite{Henson1971}.

Another situation in which the zero-one law of Corollary~\ref{cor:01law:random-subgraph} follows from a simpler result is when $G$ is a Cayley graph, possibly with additional ``decorations'' that maintain shift-invariance.
The ``decorations'' can, for instance, involve gluing a fixed finite or countably infinite graph to each vertex, or replacing each edge corresponding to a fixed generator with a copy of a fixed finite graph.  For the subgraphs of such ``decorated'' Cayley graphs, all graph-theoretic properties are shift-invariant, and hence the result follows from a version of Theorem~\ref{thm:01law:shift-invariant} in which the group~$\ZZ$ is replaced with another countable group.





Let us now give a fairly general construction of a family of graphs that have just enough symmetries for them to satisfy the hypothesis of Corollary~\ref{cor:01law:random-subgraph}.  

\begin{example}[Infinite graphs with non-trivial zero-one law]
\label{exp:graph:with-symmetries}
    We construct a countably infinite graph~$G$ in a recursive fashion.
    See Figure~\ref{fig:graph:with-symmetries} for an illustration.
    
    We start with an arbitrary finite graph~$G_0$.
    To construct $G_n$ from $G_{n-1}$, we take two copies $G_{n-1}^a$ and $G_{n-1}^b$ of $G_{n-1}$ and another arbitrary finite graph~$L_n$, all disjoint from one another.
    Then, for each vertex $v$ in $L_n$, we arbitrarily choose between connecting $v$ to either all or none of the vertices in $G_{n-1}^a\cup G_{n-1}^b$.  We may also add edges between $G_{n-1}^a$ and $G_{n-1}^b$ in a symmetric fashion: if we add an edge between $u^a$ and $v^b$, we must also add an edge between $v^a$ and $u^b$.
    We let $G$ be the inverse limit of the sequence $G_0,G_1,\ldots$ where $G_{n-1}$ is identified with one of its two copies in $G_n$.

    To see that $G$ satisfies the hypothesis of Corollary~\ref{cor:01law:random-subgraph}, let $A$ be a finite set of vertices in~$G$.
    Let $n$ be the smallest value such that $A$ is included in the vertex set of one of the copies of $G_n$ in $G_{n+1}$.  By construction, $G_{n+1}$ has an automorphism~$\theta_n$ that exchanges the two copies of $G_n$ (in particular, $\theta_n(A)\cap A=\varnothing$) and leaves the other vertices unchanged.  This automorphism in turn extends to an automorphism $\theta$ of $G$ that leaves the vertices outside $G_{n+1}$ unchanged.
    
    We conclude that for a random subgraph of $G$, all the graph-theoretic properties are deterministic.
\end{example}

\begin{figure}
    \begin{center}
        \begin{tikzpicture}[baseline,scale=0.4]   
            \graphGGGGG{0}{0}{1}{1}{}

            \begin{scope}[blue!50,dashed,thick,rounded corners=1ex,overlay]
                \draw
                    ($(aaaaa)+(-0.4\graphunit,+0.4\graphunit)$) rectangle
                    ($(aaaaa)+(+0.4\graphunit,-0.4\graphunit)$);
                \draw
                    ($(aaaaa)+(-0.6\graphunit,+0.6\graphunit)$) rectangle
                    ($(aaaab)+(+0.6\graphunit,-0.6\graphunit)$);
                \draw
                    ($(aaaaa)+(-0.8\graphunit,+0.8\graphunit)$) rectangle
                    ($(aaabb)+(+0.8\graphunit,-0.8\graphunit)$);
                \draw
                    ($(aaaaa)+(-1.0\graphunit,+1.0\graphunit)$) rectangle
                    ($(aabba)+(+1.0\graphunit,-1.0\graphunit)$);
                \draw
                    ($(aaaaa)+(-1.2\graphunit,+1.2\graphunit)$) rectangle
                    ($(abbaa)+(+1.2\graphunit,-1.2\graphunit)$);
                \draw
                    ($(aaaaa)+(-1.4\graphunit,+1.4\graphunit)$) rectangle
                    ($(bbaaa)+(+1.4\graphunit,-1.4\graphunit)$);
            \end{scope}
        \end{tikzpicture}
    \end{center}
\caption{%
    A few iterations of the recursive construction from Example~\ref{exp:graph:with-symmetries} with arbitrary choices.
    The dashed rectangles indicate the graphs $G_0$, $G_2$, \ldots, $G_5$ included in one another.
}
\label{fig:graph:with-symmetries}
\end{figure}
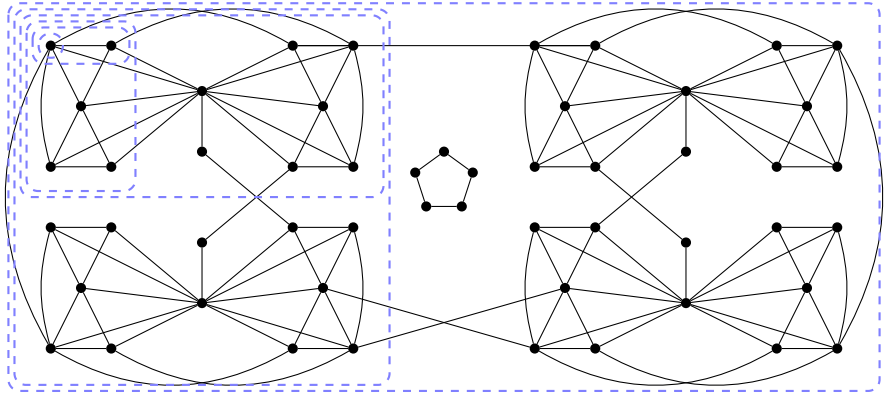


    


\subsection{Zero-one law for renormalization transformations}
\label{sec:app:renormalization}

\begin{example}[Simple majority renormalization]
\label{exp:renormalization:majority}
    Given a sequence $\config{a}=(a_k)_{k\in\ZZ}$ with values from $\{\symb{0},\symb{1}\}$, we can construct $\config{a}^0,\config{a}^1,\config{a}^2,\ldots$ as follows.  We start by setting $\config{a}^0\isdef\config{a}$.  For $n\geq 1$, we recursively define $\config{a}^n=(a^n_k)_{k\in\ZZ}$ to be the sequence in which 
    \begin{align}
        a^n_k &\isdef \maj\big(a^{n-1}_{3k-1},a^{n-1}_{3k},a^{n-1}_{3k+1}\big)
    \end{align}
    for each~$k\in\ZZ$.
    In other words, $\config{a}^n$ is obtained by grouping the elements of $\config{a}^{n-1}$ into blocks of size~$3$ (centered at positions $3k$, for $k\in\ZZ$) and replacing each block with the majority value within that block
    (see Figure~\ref{fig:renormalization:majority}).
    Let $\config{a}^*\isdef(a^0_0,a^1_0,a^2_0,\ldots)$ be the sequence  of values observed at position~$0$ during the above process.  We call $\config{a}^*$ the \emph{trace} sequence associated to~$\config{a}$.
    
    Now, let $\config{X}=(X_k)_{k\in\ZZ}$ be a sequence of \ac{iid} Bernoulli random variables with parameter~$p$.
    We are interested in the probability of the event that the trace of~$\config{X}$ eventually stabilizes at $\symb{1}$.
    
    \begin{enumerate}[label={(\alph*)}]
        \item \label{item:renormalization:majority:01law}
            Let us use Theorem~\ref{thm:main:iid} to argue that whether the trace of $\config{X}$ eventually stabilizes at~$\symb{1}$ or not is deterministic.
    
            First, note that $a^n_k$ is function of a finite number of elements from $\config{a}$.  Let $B^n_k\subseteq\ZZ$ denote the set of indices of those elements.
            For instance, $B^0_0=\{0\}$, $B^1_0=\{-1,0,1\}$, $B^2_0=\{-4,-3,\ldots,4\}$, and so on.
            Observe that $B^n_k$ is a disjoint union of the three blocks $B^{n-1}_{3k-1},B^{n-1}_{3k},B^{n-1}_{3k+1}$.
        
            Let $E_{\symb{1}}\subseteq\{\symb{0},\symb{1}\}^\ZZ$ denote the event that the trace eventually stabilizes at $\symb{1}$, respectively, that is,
            \begin{align}
                E_{\symb{1}} &\isdef \{\config{a}: \exists m\forall n\geq m, a^n_0=\symb{1}\} \;.
            \end{align}
            Given a finite set $J\subseteq\ZZ$, we let $n$ be the smallest number such that $B^n_0\supseteq J$.  Now, let $\pi\colon\ZZ\to\ZZ$ be a permutation that exchanges the blocks $B^n_0$ and $B^n_1$ and leaves every other element unchanged.  Clearly, $\pi(J)\cap J=\varnothing$.  Moreover, note that exchanging $B^n_0$ and $B^n_1$ only switches the values of $a^n_0$ and $a^n_1$ and preserves the majority value among $a^n_{-1},a^n_0,a^n_1$.  In particular, it keeps the values of $\config{a}^m$ for $m\geq n+1$ unchanged.  This means that $\pi$ is an (injective) positional symmetry of~$E_{\symb{1}}$.
            Therefore, $\PP(\config{X}\in E_{\symb{1}})\in\{0,1\}$ by Theorem~\ref{thm:main:iid}.
            
            By symmetry, we also have $\PP(\config{X}\in E_{\symb{0}})\in\{0,1\}$, where $E_{\symb{0}}$ is the event that the trace eventually stabilizes at~$\symb{0}$.
            In conclusion, we have narrowed down the eventual behaviour of the trace sequence into three possibilities:
            \begin{itemize}
                \item Almost surely, the trace of $\config{X}$ stabilizes at~$\symb{1}$.
                \item Almost surely, the trace of $\config{X}$ stabilizes at~$\symb{0}$.
                \item Almost surely, the trace of $\config{X}$ does not stabilize.
            \end{itemize}
        \item \label{item:renormalization:majority:dynaics}
            In fact, this example is simple enough that we can explicitly identify the values of~$p$ for which the trace sequence stabilizes at~$\symb{0}$, $\symb{1}$, or neither.
    
            Observe that if $\config{X}$ is an \ac{iid} sequence, then so is $\config{X}^n$ for each~$n$.
            Furthermore, if $X_k$'s have Bernoulli parameter~$p$, then we have $\PP(X^n_0=\symb{1})=f^n(p)$ where
            $f(p)\isdef 3p^2(1-p)+p^3$.
            Now, note that $p=1$ is an attractive fixed point of the continuous map $f\colon[0,1]\to[0,1]$ with $(\nicefrac{1}{2},1]$ as its basin of attraction (see Figure~\ref{fig:renormalization:majority:dynamics}).
            Therefore, $\lim_{n\to\infty}\PP(X^n_0=\symb{1})=1$ when $p>\nicefrac{1}{2}$.
            In fact, since $f$ is continuously differentiable and $p=1$ is a hyperbolic fixed point with $\abs{f'(1)}<1$, for each $p>\nicefrac{1}{2}$ we have
            \begin{align}
                \sum_{n=0}^\infty \PP(X^n_0\neq\symb{1})=\sum_{n=0}^\infty \abs[\big]{1-f^n(p)} < \infty
            \end{align}
            (see \ac{eg}, Section~1.4 in the book by Devaney~\cite{Devaney2018}).
            Therefore, by the Borel-Cantelli lemma, we have
            $\PP(\config{X}\in E_{\symb{1}})=1-\PP(\text{$X^n_0\neq\symb{1}$ for infinitely many~$n$})=1$ when $p>\nicefrac{1}{2}$.
            Similarly, $\PP(\config{X}\in E_{\symb{0}})=1$ when $p<\nicefrac{1}{2}$.
            Lastly, note that when $p=\nicefrac{1}{2}$, we have $\PP(\config{X}\in E_{\symb{1}})=\PP(\config{X}\in E_{\symb{0}})$ by symmetry.
            The zero-one law of part~\ref{item:renormalization:majority:01law} thus implies $\PP(\config{X}\in E_{\symb{1}})=\PP(\config{X}\in E_{\symb{0}})=0$ in this case.
            \qedhere
    \end{enumerate}
\end{example}

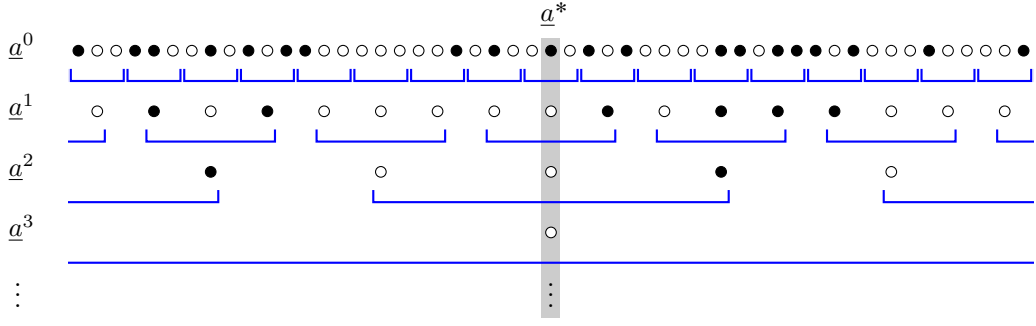
\begin{figure}
    \begin{center}
        \begin{tikzpicture}[baseline,xscale=0.25,yscale=0.8]   
            \begin{scope}
                \clip (-25.5,0.5) rectangle (25.5,-4.4);

                \fill[gray!40] (-0.5,0.2) rectangle (0.5,-5);
                
                \foreach \i in {-40,-39,...,40} {
                    \pgfmathparse{int(\rnconfigI[\i+40])}
    			    \edef\myval{\pgfmathresult}
                    \ifnumequal{\myval}{1}{
                        \node[stateI] (i0.\i) at (\i,0) {};
                    }{
                        \node[stateO] (i0.\i) at (\i,0) {};
                    }
                }
                \foreach \i in {-13,-12,...,13} {
                    \pgfmathparse{int(\rnconfigII[\i+13])}
    			    \edef\myval{\pgfmathresult}
                    \ifnumequal{\myval}{1}{
                        \node[stateI] (i1.\i) at (3*\i,-1) {};
                    }{
                        \node[stateO] (i1.\i) at (3*\i,-1) {};
                    }
                }
                \foreach \i in {-4,-3,...,4} {
                    \pgfmathparse{int(\rnconfigIII[\i+4])}
    			    \edef\myval{\pgfmathresult}
                    \ifnumequal{\myval}{1}{
                        \node[stateI] (i2.\i) at (9*\i,-2) {};
                    }{
                        \node[stateO] (i2.\i) at (9*\i,-2) {};
                    }
                }
                \foreach \i in {-1,0,1} {
                    \pgfmathparse{int(\rnconfigIV[\i+1])}
    			    \edef\myval{\pgfmathresult}
                    \ifnumequal{\myval}{1}{
                        \node[stateI] (i3.\i) at (27*\i,-3) {};
                    }{
                        \node[stateO] (i3.\i) at (27*\i,-3) {};
                    }
                }
                \foreach \i in {-13,-12,...,13} {
                    \draw[thick,blue] (-3*\i-1-0.4,-0.3) -- ++(0,-0.2) -- ++(2+2*0.4,0) -- ++(0,0.2);
                }
                \foreach \i in {-4,-3,...,4} {
                    \draw[thick,blue] (-9*\i-3-0.4,-1-0.3) -- ++(0,-0.2) -- ++(6+2*0.4,0) -- ++(0,0.2);
                }
                \foreach \i in {-1,0,1} {
                    \draw[thick,blue] (-27*\i-9-0.4,-2-0.3) -- ++(0,-0.2) -- ++(18+2*0.4,0) -- ++(0,0.2);
                }
                \draw[thick,blue] (-81*0-27-0.4,-3-0.3) -- ++(0,-0.2) -- ++(54+2*0.4,0) -- ++(0,0.2);
            \end{scope}

            \node at (0.2,0.6) {$\config{a}^*$};
            \node at (0,-4+0.1) {$\vdots$};

            \begin{scope}[overlay]
                \foreach \j in {0,1,2,3} {
                    \node at (-28,-\j+0.07) {$\config{a}^\j$};
                }
                \node at (-28-0.2,-4+0.1) {$\vdots$};
            \end{scope}
        \end{tikzpicture}
    \end{center}
\caption{%
    The renormalization process of Example~\ref{exp:renormalization:majority}.
}
\label{fig:renormalization:majority}
\end{figure}

\begin{figure}[!h]
    \begin{center}
        \begin{tikzpicture}[baseline,>=stealth,scale=5]
            \draw[->] (-0.1,0) -- (1.1,0) node[right,overlay] {$p$};
            \draw[->] (0,-0.1) -- (0,1.1); 

            \draw[gray,dashed] (0.5,-\xplottick) node[below,black] {\footnotesize $0.5$}
                -- (0.5,0.5) -- (-\xplottick,0.5) node[left,black] {\footnotesize $0.5$};

            \draw[gray,dashed] (1,-\xplottick) node[below,black] {\footnotesize $1$}
                -- (1,1) -- (-\xplottick,1) node[left,black] {\footnotesize $1$};

            \draw[gray] (-0.07,-0.07) -- (1.05,1.05);

            \draw[domain=-0.07:1.1,smooth,variable=\p,ultra thick,blue]
                plot ({\p},{3*\p*\p*(1-\p)+\p*\p*\p}) node[right,black,overlay] {$f(p)$};

            \draw[red,thick,rounded corners=0.5pt] (0.6,0) -- (0.6,0.648) -- (0.648,0.648)
                -- (0.648,0.716) -- (0.716,0.716) -- (0.716,0.803) -- (0.803,0.803)
                -- (0.803,0.899) -- (0.899,0.899) -- (0.899,0.972) -- (0.972,0.972)
                -- (0.972,0.998) -- (0.998,0.998);
        \end{tikzpicture}

    \end{center}
\caption{%
    The dynamics of $p\mapsto f(p)$ in Example~\ref{exp:renormalization:majority}.
    As $n\to\infty$, we have $f^n(p)\to 1$ if $p>\nicefrac{1}{2}$ and $f^n(p)\to 0$ if $p<\nicefrac{1}{2}$.
}
\label{fig:renormalization:majority:dynamics}
\end{figure}
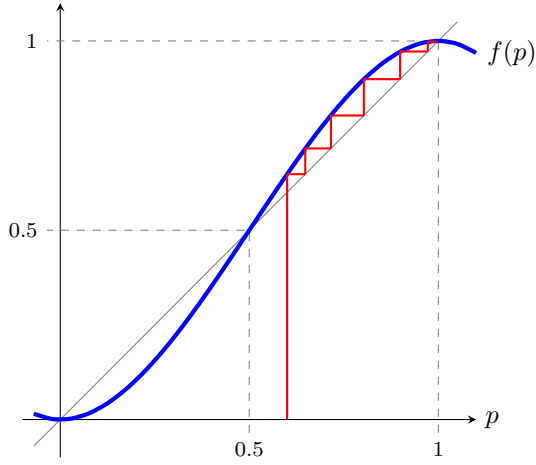

The argument in part~\ref{item:renormalization:majority:01law} of Example~\ref{exp:renormalization:majority} can be extended to a general setting in which the alphabet~$\{\symb{0},\symb{1}\}$ is replaced with an arbitrary measurable space and the majority renormalization map is replaced with any simple symmetric renormalization map, as formulated below.

Let $M$ be a measurable space.
A map $T\colon M^\ZZ\to M^\ZZ$ is called a \emph{renormalization map} if there exist integers $\ell,r\geq 0$ and a measurable map $\varphi\colon M^{2\ell+2r+1}\to M$ such that
\begin{align}
    T(\config{a})_k &= \varphi\big(a_{(2\ell+1)k-\ell-r},a_{(2\ell+1)k-\ell-r+1},\ldots,a_{(2\ell+1)k+\ell+r}\big)
\end{align}
for every $\config{a}\in M^\ZZ$ and $k\in\ZZ$.
The map $\varphi$ is called the \emph{local rule} of~$T$, the value $2\ell+1$ its \emph{block size}, and $r$ its \emph{interaction range}.
When $r=0$, we say that $T$ is \emph{simple}.
By a \emph{positional symmetry} of a map $\varphi\colon M^n\to M$, we mean a permutation $\pi$ of $\{1,2,\ldots,n\}$ such that $\varphi(a_{\pi(1)},a_{\pi(2)},\ldots,a_{\pi(n)})=\varphi(a_1,a_2,\ldots,a_n)$ for all $a_1,a_2,\ldots,a_n\in M$.  
We say that a positional symmetry $\pi$ of $\varphi$ \emph{moves} its $i$th argument if $\pi(i)\neq i$.
The \emph{trace} of a sequence $\config{a}\in M^\ZZ$ under a map $T\colon M^\ZZ\to M^\ZZ$ is the sequence $\config{a}^T\isdef(T^n(\config{a})_0)_{n\in\NN_0}$.

\begin{corollary}<thm:main:iid>[Zero-one law for the trace of simple symmetric renormalization maps]
\label{cor:01law:renormalization}
    Let $M$ be a measurable space and $T\colon M^\ZZ\to M^\ZZ$ a simple renormalization map with local rule $\varphi\colon M^{2\ell+1}\to M$.
    Suppose that $\varphi$ has a positional symmetry that moves its central argument.
    Let $\config{X}=(X_k)_{k\in\ZZ}$ be a sequence of \ac{iid} random variables with values from~$M$, and let $\config{X}^T$ be the trace of $\config{X}$ under~$T$.
    Then, for every tail event $E\subseteq M^{\NN_0}$, we have $\PP(\config{X}^T\in E)\in\{0,1\}$.
\end{corollary}
\noindent The proof is analogous to the argument given in part~\ref{item:renormalization:majority:01law} of Example~\ref{exp:renormalization:majority}.

The argument in part~\ref{item:renormalization:majority:dynaics} of Example~\ref{exp:renormalization:majority} suggests that, in Corollary~\ref{cor:01law:renormalization}, the symmetry assumption on the local rule might be superfluous.
It would be interesting to know whether Corollary~\ref{cor:01law:renormalization} would still hold if we relaxed the \ac{iid} condition or if we allowed $T$ to be non-simple. 

\section{Relaxing the \ac{iid} condition}
\label{sec:non-iid}

In Theorem~\ref{thm:main:iid}, we exploited the following two properties of the random variables:
\begin{enumerate}[label={\textup{(\roman*)}}]
    \item $\config{X}_J$ and $\config{X}^\pi_J$ are independent,
    \item The families
        $\config{X}$ and $\config{X}^\pi$ have the same distribution.
\end{enumerate}
The former was ensured by the independence of $X_k$'s and the assumption $J\cap\pi(J)=\varnothing$; the latter was guaranteed by the injectivity of $\pi$ and the \ac{iid} assumption.
The main result of this section (Theorem~\ref{thm:main:non-iid}) is a generalization of Theorem~\ref{thm:main:iid} in which ``independence'' is replaced with ``almost independence'' and ``having the same distribution'' is replaced with ``having close distributions.''
We formulate this result in Section~\ref{sec:non-iid:general:statement} and prove it in Section~\ref{sec:non-iid:general:proof}.
In Sections~\ref{sec:non-iid:independent} and~\ref{sec:non-iid:k-mixing}, we prove Theorems~\ref{thm:main:independent} and~\ref{thm:main:k-mixing} as corollaries of the general result.

\subsection{General result: statement}
\label{sec:non-iid:general:statement}

Before stating the result, let us provide mathematical formulations of what we mean by ``almost independence'' and ``having close distributions.''



\begin{definition}[Almost independence]
\label{def:almost-independence}
    Let $X$ and $Y$ be two random variables taking values in measurable spaces $\mathcal{X}$ and $\mathcal{Y}$, respectively.  Let $\varepsilon>0$.
    We say that $X$ and $Y$ are \emph{$\varepsilon$-dependent} if
    \begin{align}
        \sup_{A,B} \abs[\big]{\PP(\text{$X\in A$ and $Y\in B$}) - \PP(X\in A)\PP(Y\in B)} &< \varepsilon \;,
    \end{align}
    where the supremum is over all measurable sets $A\subseteq\mathcal{X}$ and $B\subseteq\mathcal{Y}$.
\end{definition}
\noindent Thus, $X$ and $Y$ are independent if and only if they are $\varepsilon$-dependent for every~$\varepsilon>0$.  Let us observe that if $X$ and $Y$ are $\varepsilon$-dependent, then so are $f(X)$ and $g(Y)$ for every two measurable functions $f$ and $g$.


We use ``closeness'' in the sense of total variation distance, although a somewhat weaker, more technical notion would be sufficient for the proof.
\begin{definition}[Closeness of distributions]
\label{def:closeness-of-distributions}
    Let $\mu$ and $\nu$ be two probability measures on the same measurable space~$\mathcal{X}$.
    The \emph{total variaiton distance} between $\mu$ and $\nu$ is defined as
    \begin{align}
        \norm{\mu-\nu}_\TV &\isdef \sup_A\abs{\mu(A)-\nu(A)} \;,
    \end{align}
    where the supremum is taken over all measurable sets $A\subseteq\mathcal{X}$.
    Given $\delta>0$, we say that $\mu$ and $\nu$ are \emph{$\delta$-close} if
    $\norm{\mu-\nu}_\TV<\delta$.
\end{definition}
\noindent Note that if the distributions of two random variables $X$ and $Y$ are $\delta$-close, then so are the distributions of $f(X)$ and $f(Y)$ for every measurable function $f$.

We are now ready to state the most general result of this paper.
By a \emph{finitary observable} we mean a measurable map with a finite range.
\begin{theorem}[Zero-one law for non-\ac{iid} RVs]
\label{thm:main:non-iid}
    Let $\config{X}\isdef (X_k)_{k\in\Gamma}$ be a countable family of random variables with values from a measurable space~$M$.
    Consider an event $E\subseteq M^\Gamma$,
    and suppose that for every finite set $J\subseteq\Gamma$, every finitary observable $\varphi\colon M^J\to\Sigma$, and every $\varepsilon,\delta>0$, the event~$E$ has a positional symmetry $\pi\colon\Gamma\to\Gamma$ such that
    \begin{enumerate}[label={\textup{(\roman*)}}]
        \item \label{item:thm:non-iid:almost-independence}
            $\varphi(\config{X}_J)$ and $\varphi(\config{X}^\pi_J)$ are $\varepsilon$-dependent,
        \item \label{item:thm:non-iid:closeness}
            The distributions of $\config{X}$ and $\config{X}^\pi\isdef(X_{\pi(k)})_{k\in\Gamma}$ are $\delta$-close.
    \end{enumerate}
    Then, $\PP(\config{X}\in E)\in\{0,1\}$.
\end{theorem}

\subsection{General result: proof}
\label{sec:non-iid:general:proof}

In this section, we prove Theorem~\ref{thm:main:non-iid}.
The proof follows the same pattern as the one for Theorem~\ref{thm:main:iid}, but relies on an approximate version of O'Connell's inequality for almost independent random variables.
In addition, we need two standard lemmas that, for discrete random variables, translate almost independence and closeness of distributions into statements about mutual information.

\begin{lemma}[Almost independence in terms of mutual information]
\label{lem:almost-independence:equivalence}
    Let $\mathcal{X}$ and $\mathcal{Y}$ be finite sets.
    For every $\gamma>0$, there exists an $\varepsilon>0$ such that whenever two random variables $X$ and $Y$, taking values in $\mathcal{X}$ and $\mathcal{Y}$, are $\varepsilon$-dependent (in the sense of Definition~\ref{def:almost-independence}), we have $I(X:Y)<\gamma$.
\end{lemma}
\noindent Lemma~\ref{lem:almost-independence:equivalence} follows from Lemma~\ref{lem:mutual-information:continuity} (see below) if we chose $X'$ and $Y'$ to be independent but have the same marginal distributions as $X$ and $Y$.  To get a concrete bound, we can use the ``reverse Pinsker inequality'' of~Györfi~\cite{CT2006a,SV2015}.  The statement of Lemma~\ref{lem:almost-independence:equivalence} fails for random variables with non-finite range.  The converse is however true in general, thanks to Pinsker's inequality (see \ac{eg}, the book of Cover and Thomas~\cite[Lem.~11.6.1]{CT2006}): small mutual information implies almost independence.

\begin{lemma}[Uniform continuity of mutual information]
\label{lem:mutual-information:continuity}
    Let $\mathcal{X}$ and $\mathcal{Y}$ be finite sets.
    For every $\beta>0$, there exists a $\delta>0$ such that whenever the distributions of two pairs $(X,Y)$ and $(X',Y')$ of random variables, taking values in $\mathcal{X}\times\mathcal{Y}$, are $\delta$-close (in the sense of Definition~\ref{def:closeness-of-distributions}), we have $\abs[\big]{I(X':Y')-I(X:Y)}<\beta$.
\end{lemma}
\noindent Lemma~\ref{lem:mutual-information:continuity} simply expresses the uniform continuity of $I(X:Y)$ on the joint distribution of $X$ and~$Y$.  The statement does not hold for random variables with non-finite range, but mutual information remains lower semi-continuous on the joint distribution.

\begin{lemma}[O'Connell's inequality for almost independent RVs]
\label{lem:OConnell:approximate}
    Let $\gamma>0$.
    Let $V,W_1,W_2,\ldots,W_n$ be discrete random variables and suppose that $I(W_i:(W_1,W_2,\ldots,W_{i-1}))<\gamma$ for $i=2,3,\ldots,n$.  Then
	\begin{align}
		H(V) &> \sum_{i=1}^{n} I(V:W_i) - n\gamma
	\end{align}
\end{lemma}
\begin{proof}
    We have
    \begin{align}
        H(V) 
        &\geq I\big(V : (W_1,W_2,\ldots,W_n)\big) \\
        &=
            \sum_{i=1}^n I\big(V: W_i\given[\big] (W_1,W_2,\ldots,W_{i-1})\big)
            && \text{(chain rule)} \\
        &=
            I(V:W_1) + \begin{multlined}[t][0.35\linewidth]
                \sum_{i=2}^n\Big[I\big(W_i:(V,W_1,\ldots,W_{i-1})\big) \\
                    - I\big(W_i:(W_1,W_2,\ldots,W_{i-1})\big)\Big]
            \end{multlined}
            && \text{(chain rule)} \\
        &\geq
            I(V:W_1) + \sum_{i=2}^n\big[I(W_i:V) - \gamma\big]
            && \text{(data-processing + hypothesis)} \\
        &=
            \sum_{i=1}^n I(W_i:V) - (n-1)\gamma \;,
    \end{align}
    which gives the result.
\end{proof}
    

\begin{proof}[Proof of Theorem~\ref{thm:main:non-iid}]
    Let $Y\isdef\indicator{E}(\config{X})$.
    As in the proof of Theorem~\ref{thm:main:iid}, we show that $Y$ is independent of $\config{X}_K$ for every finite $K\subseteq\Gamma$.  The result would then follow as in Section~\ref{sec:intro:oconnel}.

    So, let $K\subseteq\Gamma$ be finite.
    Let $C\subseteq M^K$ be a measurable set and $\gamma,\beta>0$ be arbitrary.
    \begin{claim*}
        The event $E$ has a sequence $\pi_1,\pi_2,\ldots$ of positional symmetries
        such that
        \begin{enumerate}[label={(\roman*)}]
            \item \label{item:non-iid:claim:almost-independence}
                $I\big(\indicator{C}(\config{X}^{\pi_n}_K):\big(\indicator{C}(\config{X}^{\pi_1}_K),\indicator{C}(\config{X}^{\pi_2}_K),\ldots,\indicator{C}(\config{X}^{\pi_{n-1}}_K)\big)\big)<\gamma$ for each $n=2,3,\ldots$.
            \item \label{item:non-iid:claim:closeness}
                $\abs[\Big]{I\big(Y:\indicator{C}(\config{X}^{\pi_n}_K)\big) - I\big(Y:\indicator{C}(\config{X}_K)\big)}<\beta$ for each $n=1,2,\ldots$.
        \end{enumerate}
    \end{claim*}
    \begin{claimproof}
        Applying Lemma~\ref{lem:mutual-information:continuity} with $\mathcal{X}=\mathcal{Y}=\{0,1\}$, we get a value $\delta>0$ corresponding to~$\beta$.
        We choose $\pi_1,\pi_2,\ldots$ recursively.
        \begin{itemize}
            \item Let $\pi_1$ be the identity map on~$\Gamma$ so that~\ref{item:non-iid:claim:closeness} is trivially satisfied for $n=1$.
            \item For $n\geq 2$, suppose that we have already chosen $\pi_1,\pi_2,\ldots,\pi_{n-1}$.
            Applying Lemma~\ref{lem:almost-independence:equivalence} with $\mathcal{X}=\{0,1\}$ and $\mathcal{Y}=\{0,1\}^{n-1}$, we obtain a value $\varepsilon>0$ corresponding to~$\gamma$.
            Set $J\isdef\bigcup_{i=1}^{n-1}\pi_i(K)$ and consider the finitary observable $\varphi\colon M^J\to\{0,1\}^{n-1}$ defined by
            \begin{align}
                \varphi(\config{a}_J) &\isdef \big(\indicator{C}(\config{a}^{\pi_1}_K),\indicator{C}(\config{a}^{\pi_2}_K),\ldots,\indicator{C}(\config{a}^{\pi_{n-1}}_K)\big) \;.
            \end{align}
            Let $\pi_n$ be a positional symmetry of $E$ such that $\varphi(\config{X}_J)$ and $\varphi(\config{X}^{\pi_n}_J)$ are $\varepsilon$-dependent and the distributions of $\config{X}$ and $\config{X}^{\pi_n}$ are $\delta$-close.

            First, observe that $S\isdef\indicator{C}(\config{X}^{\pi_n}_K)$ and $T\isdef\big(\indicator{C}(\config{X}^{\pi_1}_K),\indicator{C}(\config{X}^{\pi_2}_K),\ldots,\indicator{C}(\config{X}^{\pi_{n-1}}_K)\big)$ are measurable functions of $\varphi(\config{X}^{\pi_n}_J)$ and $\varphi(\config{X}_J)$, respectively.
            Since $\varphi(\config{X}_J)$ and $\varphi(\config{X}^{\pi_n}_J)$ are $\varepsilon$-dependent,
            so are $S$ and~$T$ (see the comment after Definition~\ref{def:almost-independence}).
            It follows that $I(S:T)<\gamma$,
            hence~\ref{item:non-iid:claim:almost-independence} is satisfied.

            Next, note that since the distributions of $\config{X}$ and $\config{X}^{\pi_n}$ are $\delta$-close, so are the joint distributions of the pairs $\big(\indicator{E}(\config{X}),\indicator{C}(\config{X}_K)\big)$ and $\big(\indicator{E}(\config{X}^{\pi_n}),\indicator{C}(\config{X}^{\pi_n}_K)\big)$ (see the comment after Definition~\ref{def:closeness-of-distributions}).  Hence,
            \useshortskip
            \begin{align}
                \abs[\Big]{I\big(\indicator{E}(\config{X}^{\pi_n}):\indicator{C}(\config{X}^{\pi_n}_K)\big) - I\big(\indicator{E}(\config{X}):\indicator{C}(\config{X}_K)\big)} &< \beta \;.
            \end{align}
            But $\pi_n$ is a positional symmetry of $E$, hence $\indicator{E}(\config{X}^{\pi_n})=\indicator{E}(\config{X})$.  Therefore,
            \begin{align}
                \abs[\Big]{I\big(\indicator{E}(\config{X}):\indicator{C}(\config{X}^{\pi_n}_K)\big) - I\big(\indicator{E}(\config{X}):\indicator{C}(\config{X}_K)\big)} &< \beta \;,
            \end{align}
            which means~\ref{item:non-iid:claim:closeness} is satisfied, too.
        \end{itemize}
        Thus, the claim holds.
    \end{claimproof}

    Now, by O'Connell's inequality for almost independent random variables (Lemma~\ref{lem:OConnell:approximate}), for every $n$ we have
    \begin{align}
        H(Y) &\geq \sum_{i=1}^n I\big(Y:\indicator{C}(\config{X}^{\pi_i}_K)\big) - n\gamma
            > n I\big(Y:\indicator{C}(\config{X}_K\big) - n\gamma -n\beta \;.
    \end{align}
    Since $H(Y)$ is finite and $n$ is arbitrary, we obtain $I\big(Y:\indicator{C}(\config{X}_K)\big)<\gamma+\beta$.
    But $\gamma>0$ and $\beta>0$ are also arbitrary, thus $I\big(Y:\indicator{C}(\config{X}_K)\big)=0$.
    Therefore, $Y$ is independent of~$\indicator{C}(\config{X}_K)$.
    Lastly, since $C$ is arbitrary, we find that $Y$ is independent of $\config{X}_K$, as claimed.
    This concludes the proof of the theorem.
\end{proof}

\subsection{Independent but not identically distributed RVs}
\label{sec:non-iid:independent}

\begin{proof}[Proof of Theorem~\ref{thm:main:independent}]
    We verify that the hypothesis of Theorem~\ref{thm:main:non-iid} is satisfied.
    Let $J\subseteq\Gamma$ be a finite set, $\varphi\colon M^J\to\Sigma$ a finitary observable, and $\varepsilon,\delta>0$.
    By assumption, there exists an injective positional symmetry $\pi\colon\Gamma\to\Gamma$ such that $\pi(J)\cap J=\varnothing$ and $\sum_{k\in\Gamma}\norm{p_{\pi(k)}-p_k}_\TV<\delta$.
    The independence of $X_k$'s along with the fact that $\pi(J)$ and $J$ are disjoint ensure that $\config{X}_J$ and $\config{X}^\pi_J$ are independent, thus condition~\ref{item:thm:non-iid:almost-independence} is satisfied.
    To see that condition~\ref{item:thm:non-iid:closeness} is also satisfied, first note that, since $\pi$ is injective, the family $\config{X}^\pi=(X_{\pi(k)})_{k\in\Gamma}$ is also an \ac{iid} family.
    The claim thus follows once we note that whenever $\mu\isdef\bigotimes_{k\in\Gamma}p_k$ and $\nu\isdef\bigotimes_{k\in\Gamma}q_k$ are two product measures on the same product space, we have
    \begin{align}
        \norm{\mu-\nu}_\TV &\leq \sum_{k\in\Gamma}\norm{q_k-p_k}_\TV \;.
    \end{align}
    To establish the latter inequality, construct a coupling $(\config{V},\config{W})$ of $\mu$ and $\nu$ such that
    \begin{enumerate*}[label={(\arabic*)}]
        \item for each $k\in\Gamma$, the pair $(V_k,W_k)$ is an optimal coupling of $p_k$ and $q_k$, and
        \item the pairs $(V_k,W_k)$ are independent of one another.
    \end{enumerate*}
    Then, by the coupling inequality and the union bound, we have
    \begin{align}
        \norm{\mu-\nu}_\TV &\leq \PP(\config{V}\neq\config{W})
            \leq \sum_{k\in\Gamma} \PP(V_k\neq W_k) = \sum_{k\in\Gamma}\norm{q_k-p_k}_\TV \;.
    \end{align}
    (See the book by Lindvall~\cite{Lindvall2002} for details.)
    This concludes the proof.
\end{proof}

\begin{remark}[Theorem~\ref{thm:main:independent} is not sharp]
\label{rem:main:independent:not-necessary}
    In Theorem~\ref{thm:main:independent}, the condition on $p_k$'s is not necessary in order for $E$ to be deterministic.
    To see this, note that if $\mu\isdef\bigotimes_{k\in\Gamma}p_k$ and $\nu\isdef\bigotimes_{k\in\Gamma}q_k$ are two product measures on $M^\Gamma$ such that $q_k=p_k$ for all but finitely many $k$'s and $q_k\ll p_k$ for the rest, then $\nu\ll\mu$.  In particular, $\nu(E)=0$ (resp.~$\nu(E)=1$) whenever $\mu(E)=0$ (resp.~$\mu(E)=1$).
    However, it is very well possible that the family $(p_k)_{k\in E}$ satisfies the hypothesis of Theorem~\ref{thm:main:independent} while $(q_k)_{k\in E}$ does not.
    For instance, consider the scenario in which
    \begin{itemize}
        \item $p_k=p$ for all $k$,
        \item $q_k=p$ for all $k\neq k_0$ and $q_{k_0}=p'$,
    \end{itemize}
    where $k_0$ is an arbitrary element of~$\Gamma$, and $p,p'$ are distinct measures on~$M$ such that $p'\ll p$.
    Then, $\sum_{k\in\Gamma}\norm{q_{\pi(k)}-q_k}_\TV$ is bounded away from zero for all maps $\pi\colon\Gamma\to\Gamma$ such that $\pi(k_0)\neq k_0$.  Hence, the hypothesis fails for $J\isdef\{k_0\}$.  
\end{remark}

\subsection{Asymptotically independent RVs}
\label{sec:non-iid:k-mixing}
In this section, we prove Theorem~\ref{thm:main:k-mixing} as a corollary of Theorem~\ref{thm:main:non-iid}.  But before doing so, let us recall the definition of the Kolmogorov mixing property.

An event $V\subseteq M^\Gamma$ is called a \emph{cylinder} event if it is measurably determined by the coordinates inside a finite set~$J\subseteq\Gamma$, that is, $V=\{\config{a}\in M^\Gamma: \config{a}_J\in V_0\}$ for some measurable set $V_0\subseteq M^J$.

\begin{definition}[Kolmogorov mixing property]
    A countable family $\config{X}=(X_n)_{n\in\Gamma}$ of random variables with values from a measurable space $M$ is said to have the \emph{Kolmogorov mixing property} (or has \emph{short-range correlations}) if for every cylinder event $V\subseteq M^\Gamma$ and every $\varepsilon>0$, there exists a finite set $K\subseteq\Gamma$ such that
    \useshortskip
    \begin{align}
        \sup_W\abs[\big]{\PP(\config{X}\in V\cap W) - \PP(\config{X}\in V)\PP(\config{X}\in W)} &< \varepsilon \;,
    \end{align}
    where the supremum is over all sets $W\subseteq M^\Gamma$ that are measurably determined by the coordinates outside~$K$. 
\end{definition}
\noindent The families with the Kolmogorov mixing property are precisely those for which every tail event is deterministic~\cite[Prop.~7.9]{Georgii2011}.
They include
\begin{enumerate}[label={(\alph*)}]
    \item Every recurrent Markov chain $(X_n)_{n\geq 0}$ with a countable state space that either has a deterministic initial state~$X_0$ or is irreducible and aperiodic~\cite{BF1964}.
    \item 
        Every Gibbs random field 
        that has an extremal distribution~\cite[Thm.~7.7]{Georgii2011}.
    \item Every determinantal Bernoulli process~\cite{Lyons2003}, and more generally, every negatively associated Bernoulli process with summable covariances~\cite{ABZ2023}. 
    \item Every negatively associated Gaussian process~\cite{ABZ2023}.
\end{enumerate}

\begin{proof}[Proof of Theorem~\ref{thm:main:k-mixing}]
    We verify the hypothesis of Theorem~\ref{thm:main:non-iid}.
    Let $J\subseteq\Gamma$ be a finite set, $\varphi\colon M^J\to\Sigma$ a finitary observable, and $\varepsilon,\delta>0$.
    For each $A\subseteq\Sigma$, consider the cylinder set $V_A\isdef\{\config{x}\in M^\Gamma: \varphi(\config{x}_J)\in A\}$, and let $K_A\supseteq J$ be a finite set corresponding to $V_A$ and $\varepsilon$ as in the definition of the Kolmogorov mixing property.  Let $K\isdef\bigcup_{A\subseteq\Sigma}K_A$, and note that $K$ is again finite because $\Sigma$ is finite.
    
    By assumption, $E$ has a positional symmetry $\pi$ such that $\pi(K)\cap K=\varnothing$ and $\config{X}$ and $\config{X}^\pi$ have the same distribution.  This $\pi$ automatically satisfies condition~\ref{item:thm:non-iid:closeness} in Theorem~\ref{thm:main:non-iid}.  To see that condition~\ref{item:thm:non-iid:almost-independence} is also satisfied, note that $J\subseteq K$ and $\pi(J)\subseteq\Gamma\setminus K$.  Hence, for every $A,B\subseteq\Sigma$, we have
    \begin{align}
        \abs[\Big]{%
            \PP\big(\text{$\varphi(\config{X}_J)\in A$ and $\varphi(\config{X}^\pi_J)\in B$}\big)
            - \PP\big(\varphi(\config{X}_J)\in A\big)\PP\big(\varphi(\config{X}^\pi_J)\in B\big)
        } &< \varepsilon  
    \end{align}
    because 
    $\{\varphi(\config{X}^\pi_J)\in B\}$ is determined by the coordinates outside $K$, which are all outside $K_A$.
    Since $\Sigma$ is finite, we obtain that $\varphi(\config{X}_J)$ and $\varphi(\config{X}^\pi_J)$ are $\varepsilon$-dependent, hence condition~\ref{item:thm:non-iid:almost-independence} is satisfied.
    The result thus follows from Theorem~\ref{thm:main:non-iid}.
\end{proof}

\bibliographystyle{plainurl}
\bibliography{bibliography}

\begin{thebibliography}{10}

\bibitem{AP1979}
D.~Aldous and J.~Pitman.
\newblock On the zero-one law for exchangeable events.
\newblock {\em The Annals of Probability}, 7(4), 1979.
\newblock \href {https://doi.org/10.1214/aop/1176994992}
  {\path{doi:10.1214/aop/1176994992}}.

\bibitem{Aldous1985}
D.~J. Aldous.
\newblock {\em Exchangeability and related topics}, pages 1--198.
\newblock Springer, 1985.
\newblock \href {https://doi.org/10.1007/bfb0099421}
  {\path{doi:10.1007/bfb0099421}}.

\bibitem{ABZ2023}
K.~Alishahi, M.~Barzegar, and M.~Zamani.
\newblock On tail triviality of negatively dependent stochastic processes.
\newblock {\em The Annals of Probability}, 51(4), 2023.
\newblock \href {https://doi.org/10.1214/23-aop1626}
  {\path{doi:10.1214/23-aop1626}}.

\bibitem{AB2015}
D.~Alvarez-Melis and T.~Broderick.
\newblock A translation of ``{T}he characteristic function of a random
  phenomenon'' by {B}runo de~{F}inetti, 2015.
\newblock \href {https://arxiv.org/abs/1512.01229} {\path{arXiv:1512.01229}}.

\bibitem{ABB+2004}
S.~Artstein, K.~Ball, F.~Barthe, and A.~Naor.
\newblock Solution of {S}hannon's problem on the monotonicity of entropy.
\newblock {\em Journal of the American Mathematical Society}, 17(4):975--982,
  2004.
\newblock \href {https://doi.org/10.1090/s0894-0347-04-00459-x}
  {\path{doi:10.1090/s0894-0347-04-00459-x}}.

\bibitem{Austin2008}
T.~Austin.
\newblock On exchangeable random variables and the statistics of large graphs
  and hypergraphs.
\newblock {\em Probability Surveys}, 5, 2008.
\newblock \href {https://doi.org/10.1214/08-ps124}
  {\path{doi:10.1214/08-ps124}}.

\bibitem{Barron1986}
A.~R. Barron.
\newblock Entropy and the central limit theorem.
\newblock {\em The Annals of Probability}, 14(1), 1986.
\newblock \href {https://doi.org/10.1214/aop/1176992632}
  {\path{doi:10.1214/aop/1176992632}}.

\bibitem{BGK2024}
M.~Berta, L.~Gavalakis, and I.~Kontoyiannis.
\newblock A third information-theoretic approach to finite de finetti theorems.
\newblock In {\em 2024 IEEE International Symposium on Information Theory
  (ISIT)}, pages 73--78. IEEE, 2024.
\newblock \href {https://doi.org/10.1109/isit57864.2024.10619572}
  {\path{doi:10.1109/isit57864.2024.10619572}}.

\bibitem{BF1964}
D.~Blackwell and D.~Freedman.
\newblock The tail $\sigma$-field of a {M}arkov chain and a theorem of {O}rey.
\newblock {\em Annals of Mathematical Statistics}, 35(3):1291--1295, 1964.
\newblock \href {https://doi.org/10.1214/aoms/1177703284}
  {\path{doi:10.1214/aoms/1177703284}}.

\bibitem{BP1972}
J.~R. Blum and R.~K. Pathak.
\newblock A note on the zero-one law.
\newblock {\em The Annals of Mathematical Statistics}, 43(3):1008--1009, 1972.
\newblock \href {https://doi.org/10.1214/aoms/1177692564}
  {\path{doi:10.1214/aoms/1177692564}}.

\bibitem{Cameron2013}
P.~J. Cameron.
\newblock The random graph.
\newblock In R.~Graham, J.~Nešetřil, and S.~Butler, editors, {\em The
  Mathematics of Paul Erdős II}, pages 353--378. Springer, 2nd edition, 2013.
\newblock \href {https://doi.org/10.1007/978-1-4614-7254-4_22}
  {\path{doi:10.1007/978-1-4614-7254-4_22}}.

\bibitem{CT2006}
T~M. Cover and J.~A. Thomas.
\newblock {\em Elements of Information Theory}.
\newblock Wiley, 2nd edition, 2006.
\newblock \href {https://doi.org/10.1002/047174882X}
  {\path{doi:10.1002/047174882X}}.

\bibitem{Csiszar1984}
I.~Csiszár.
\newblock {S}anov property, generalized {$I$}-projection and a conditional
  limit theorem.
\newblock {\em The Annals of Probability}, 12(3), 1984.
\newblock \href {https://doi.org/10.1214/aop/1176993227}
  {\path{doi:10.1214/aop/1176993227}}.

\bibitem{CS2004}
I.~Csiszár and P.~C. Shields.
\newblock Information theory and statistics: A tutorial.
\newblock {\em Foundations and Trends in Communications and Information
  Theory}, 1(4):417--528, 2004.
\newblock \href {https://doi.org/10.1561/0100000004}
  {\path{doi:10.1561/0100000004}}.

\bibitem{CT2006a}
I.~Csiszár and Z.~Talata.
\newblock Context tree estimation for not necessarily finite memory processes,
  via {BIC} and {MDL}.
\newblock {\em IEEE Transactions on Information Theory}, 52(3):1007--1016,
  2006.
\newblock \href {https://doi.org/10.1109/tit.2005.864431}
  {\path{doi:10.1109/tit.2005.864431}}.

\bibitem{Devaney2018}
R.~L. Devaney.
\newblock {\em An Introduction to Chaotic Dynamical Systems}.
\newblock CRC Press, 2018.
\newblock \href {https://doi.org/10.4324/9780429502309}
  {\path{doi:10.4324/9780429502309}}.

\bibitem{Diaconis1977}
P.~Diaconis.
\newblock Finite forms of de~{F}inetti's theorem on exchangeability.
\newblock {\em Synthese}, 36(2):271--281, 1977.
\newblock \href {https://doi.org/10.1007/bf00486116}
  {\path{doi:10.1007/bf00486116}}.

\bibitem{DF1980}
P.~Diaconis and D.~Freedman.
\newblock De {F}inetti's theorem for {M}arkov chains.
\newblock {\em The Annals of Probability}, 8(1), 1980.
\newblock \href {https://doi.org/10.1214/aop/1176994828}
  {\path{doi:10.1214/aop/1176994828}}.

\bibitem{DF1980a}
P.~Diaconis and D.~Freedman.
\newblock Finite exchangeable sequences.
\newblock {\em The Annals of Probability}, 8(4), 1980.
\newblock \href {https://doi.org/10.1214/aop/1176994663}
  {\path{doi:10.1214/aop/1176994663}}.

\bibitem{Dudley2004}
R.~M. Dudley.
\newblock {\em Real Analysis and Probability}.
\newblock Cambridge University Press, 2004.
\newblock \href {https://doi.org/10.1017/CBO9780511755347}
  {\path{doi:10.1017/CBO9780511755347}}.

\bibitem{Dynkin1978}
E.~B. Dynkin.
\newblock Sufficient statistics and extreme points.
\newblock {\em The Annals of Probability}, 6(5), 1978.
\newblock \href {https://doi.org/10.1214/aop/1176995424}
  {\path{doi:10.1214/aop/1176995424}}.

\bibitem{Fagin1976}
R.~Fagin.
\newblock Probabilities on finite models.
\newblock {\em The Journal of Symbolic Logic}, 41(1):50--58, 1976.
\newblock \href {https://doi.org/10.2307/2272945} {\path{doi:10.2307/2272945}}.

\bibitem{FK1996}
E.~Friedgut and G.~Kalai.
\newblock Every monotone graph property has a sharp threshold.
\newblock {\em Proceedings of the American Mathematical Society},
  124(10):2993--3002, 1996.
\newblock \href {https://doi.org/10.1090/s0002-9939-96-03732-x}
  {\path{doi:10.1090/s0002-9939-96-03732-x}}.

\bibitem{GK2021}
L.~Gavalakis and I.~Kontoyiannis.
\newblock An information-theoretic proof of a finite de~{F}inetti theorem.
\newblock {\em Electronic Communications in Probability}, 26(none), 2021.
\newblock \href {https://doi.org/10.1214/21-ecp428}
  {\path{doi:10.1214/21-ecp428}}.

\bibitem{GK2023}
L.~Gavalakis and I.~Kontoyiannis.
\newblock Information in probability: Another information-theoretic proof of a
  finite de~{F}inetti theorem.
\newblock In {\em Mathematics Going Forward}, pages 367--385. Springer, 2023.
\newblock \href {https://doi.org/10.1007/978-3-031-12244-6_26}
  {\path{doi:10.1007/978-3-031-12244-6_26}}.

\bibitem{Georgii1976}
H.-O. Georgii.
\newblock On canonical {G}ibbs states, symmetric and tail events.
\newblock {\em Zeitschrift für Wahrscheinlichkeitstheorie und Verwandte
  Gebiete}, 33(4):331--341, 1976.
\newblock \href {https://doi.org/10.1007/bf00534783}
  {\path{doi:10.1007/bf00534783}}.

\bibitem{Georgii1979}
H.-O. Georgii.
\newblock {\em Canonical {G}ibbs Measures}, volume 760 of {\em Lecture Notes in
  Mathematics}.
\newblock Springer, 1979.
\newblock \href {https://doi.org/10.1007/bfb0068557}
  {\path{doi:10.1007/bfb0068557}}.

\bibitem{Georgii2011}
H.-O. Georgii.
\newblock {\em Gibbs Measures and Phase Transitions}.
\newblock Walter de~Gruyter, 2nd edition, 2011.
\newblock \href {https://doi.org/10.1515/9783110250329}
  {\path{doi:10.1515/9783110250329}}.

\bibitem{HH2024}
N.~Halberstam and T.~Hutchcroft.
\newblock Proof of the {D}iaconis--{F}reedman conjecture on
  partially-exchangeable processes.
\newblock Preprint, 2024.
\newblock \href {https://arxiv.org/abs/2405.20276} {\path{arXiv:2405.20276}}.

\bibitem{Henson1971}
C.~W. Henson.
\newblock A family of countable homogeneous graphs.
\newblock {\em Pacific Journal of Mathematics}, 38(1):69--83, 1971.
\newblock \href {https://doi.org/10.2140/pjm.1971.38.69}
  {\path{doi:10.2140/pjm.1971.38.69}}.

\bibitem{HS1955}
E.~Hewitt and L.~J. Savage.
\newblock Symmetric measures on {C}artesian products.
\newblock {\em Transactions of the American Mathematical Society}, 80(2):470,
  1955.
\newblock \href {https://doi.org/10.2307/1992999} {\path{doi:10.2307/1992999}}.

\bibitem{James1996}
N.~C. James.
\newblock {\em Exchangeable Events for {M}arkov Chains}.
\newblock PhD thesis, University of California, Berkeley, 1996.

\bibitem{Johnson2004}
O.~Johnson.
\newblock {\em Information Theory and the Central Limit Theorem}.
\newblock Imperial College Press, 2004.
\newblock \href {https://doi.org/10.1142/p341} {\path{doi:10.1142/p341}}.

\bibitem{Kallenberg2005}
O.~Kallenberg.
\newblock {\em Probabilistic Symmetries and Invariance Principles}.
\newblock Probability and Its Applications. Springer, 2005.
\newblock \href {https://doi.org/10.1007/0-387-28861-9}
  {\path{doi:10.1007/0-387-28861-9}}.

\bibitem{Kingman1978}
J.~F.~C. Kingman.
\newblock Uses of exchangeability.
\newblock {\em The Annals of Probability}, 6(2), 1978.
\newblock \href {https://doi.org/10.1214/aop/1176995566}
  {\path{doi:10.1214/aop/1176995566}}.

\bibitem{Klenke2020}
A.~Klenke.
\newblock {\em Probability Theory: A Comprehensive Course}.
\newblock Springer International Publishing, 3rd edition, 2020.
\newblock \href {https://doi.org/10.1007/978-3-030-56402-5}
  {\path{doi:10.1007/978-3-030-56402-5}}.

\bibitem{Kolmogorov1956}
A.~N. Kolmogorov.
\newblock {\em Foundations of the Theory of Probability}.
\newblock Chelsea Publishing Company, 1956.
\newblock English Translation.

\bibitem{Lindvall2002}
T.~Lindvall.
\newblock {\em Lectures on the coupling method}.
\newblock Dover, 2002.

\bibitem{Linnik1959}
Ju.~V. Linnik.
\newblock An information-theoretic proof of the central limit theorem with
  lindeberg conditions.
\newblock {\em Theory of Probability \& Its Applications}, 4(3):288--299, 1959.
\newblock \href {https://doi.org/10.1137/1104028} {\path{doi:10.1137/1104028}}.

\bibitem{Lyons2003}
R.~Lyons.
\newblock Determinantal probability measures.
\newblock {\em Publications mathématiques de l’IHÉS}, 98(1):167--212, 2003.
\newblock \href {https://doi.org/10.1007/s10240-003-0016-0}
  {\path{doi:10.1007/s10240-003-0016-0}}.

\bibitem{OConnell2000}
N.~O'Connell.
\newblock Information-theoretic proof of the {H}ewitt-{S}avage zero-one law.
\newblock Technical Report HPL-BRIMS-2000-18, Hewlett-Packard Laboratories,
  Bristol, UK, 2000.

\bibitem{Pinsker1964}
M.~S. Pinsker.
\newblock {\em Information and Information Stability of Random Variables and
  Processes}.
\newblock Holden-Day, 1964.
\newblock English Translation.

\bibitem{Preston1979}
C.~Preston.
\newblock Canonical and microcanonical {G}ibbs states.
\newblock {\em Zeitschrift für Wahrscheinlichkeitstheorie und Verwandte
  Gebiete}, 46(2):125--158, 1979.
\newblock \href {https://doi.org/10.1007/bf00533255}
  {\path{doi:10.1007/bf00533255}}.

\bibitem{Russo1982}
L.~Russo.
\newblock An approximate zero-one law.
\newblock {\em Zeitschrift für Wahrscheinlichkeitstheorie und Verwandte
  Gebiete}, 61(1):129--139, 1982.
\newblock \href {https://doi.org/10.1007/bf00537230}
  {\path{doi:10.1007/bf00537230}}.

\bibitem{Renyi1961}
A.~Rényi.
\newblock On measures of entropy and information.
\newblock In {\em Proceedings of the 4th {B}erkeley {S}ymposium on
  {M}athematical {S}tatististics and {P}robability}, pages 547--561. University
  of California Press, 1961.

\bibitem{SV2015}
I.~Sason and S.~Verdu.
\newblock Upper bounds on the relative entropy and {R}ényi divergence as a
  function of total variation distance for finite alphabets.
\newblock In {\em 2015 IEEE Information Theory Workshop - Fall (ITW)}. IEEE,
  2015.
\newblock \href {https://doi.org/10.1109/itwf.2015.7360766}
  {\path{doi:10.1109/itwf.2015.7360766}}.

\bibitem{Spencer2001}
J.~Spencer.
\newblock {\em The Strange Logic of Random Graphs}, volume~22 of {\em
  Algorithms and Combinatorics}.
\newblock Springer, 2001.
\newblock \href {https://doi.org/10.1007/978-3-662-04538-1}
  {\path{doi:10.1007/978-3-662-04538-1}}.

\bibitem{Talagrand1994}
M.~Talagrand.
\newblock On {R}usso's approximate zero-one law.
\newblock {\em The Annals of Probability}, 22(3), 1994.
\newblock \href {https://doi.org/10.1214/aop/1176988612}
  {\path{doi:10.1214/aop/1176988612}}.

\bibitem{Walters1982}
P.~Walters.
\newblock {\em An Introduction to Ergodic Theory}, volume~79 of {\em Graduate
  Texts in Mathematics}.
\newblock Springer, 1982.

\end{thebibliography}

\end{document}